\documentclass[journal]{IEEEtran}
\usepackage{graphicx}
\graphicspath{{figures/}}
\usepackage{amsmath,amssymb,dsfont,verbatim}
\usepackage{array}
\usepackage[caption=false,font=footnotesize]{subfig}
\usepackage{import}
\usepackage{url, hyperref}
\usepackage{cite}
\usepackage{xcolor}
\usepackage{booktabs} \newcommand{\ra}[1]{\renewcommand{\arraystretch}{#1}}
\usepackage{tikz}
\usetikzlibrary{mindmap}

\newif\ifarx
 \arxtrue 

\usepackage{amsthm}

\DeclareMathOperator{\cw}{{\scriptstyle\mathcal{W}}}
\DeclareMathOperator{\bcw}{{\boldsymbol{\scriptstyle\mathcal{W}}}}

\DeclareMathOperator{\T}{\mathsf{T}}
\DeclareMathOperator{\E}{\mathds{E}}

\DeclareMathOperator{\w}{\boldsymbol{w}}
\DeclareMathOperator{\x}{\boldsymbol{x}}
\DeclareMathOperator{\s}{\boldsymbol{s}}
\DeclareMathOperator{\g}{\boldsymbol{g}}

\usepackage{amsthm}

\newtheorem{example}{Example}
\newtheorem{definition}{Definition}
\newtheorem{assumption}{Assumption}
\newtheorem{theorem}{Theorem}
\newtheorem{corollary}{Corollary}
\newtheorem{lemma}{Lemma}

\begin{document}%
\title{Distributed Learning in Non-Convex Environments -- Part I: Agreement at a Linear Rate }%
\author{Stefan Vlaski,~\IEEEmembership{Student Member,~IEEE,}
 				and Ali H. Sayed,~\IEEEmembership{Fellow,~IEEE}
\thanks{The authors are with the Institute of Electrical Engineering, \'{E}cole Polytechnique F\'{e}d\'{e}rale de Lausanne. S. Vlaski is also with the Department of Electrical Engineering, University of California, Los Angeles. This work was supported in part by NSF grant CCF-1524250. Emails:\{stefan.vlaski, ali.sayed\}@epfl.ch. A limited short version of this work appears in the conference publication~\cite{Vlaski19}.}}%
\maketitle
\begin{abstract}
Driven by the need to solve increasingly complex optimization problems in signal processing and machine learning, there has been increasing interest in understanding the behavior of gradient-descent algorithms in non-convex environments. Most available works on distributed non-convex optimization problems focus on the deterministic setting where exact gradients are available at each agent.
In  this work and its Part II, we consider stochastic cost functions, where exact gradients are replaced by stochastic approximations and the resulting gradient noise persistently seeps into the dynamics of the algorithm. We establish that the diffusion learning strategy continues to yield meaningful estimates non-convex scenarios in the sense that the iterates by the individual agents will cluster in a small region around the network centroid. We use this insight to motivate a short-term model for network evolution over a finite-horizon. In Part II~\cite{Vlaski19nonconvexP2} of this work, we leverage this model to establish descent of the diffusion strategy through saddle points in \( O(1/\mu) \) steps and the return of approximately second-order stationary points in a polynomial number of iterations.
\end{abstract}
\section{Introduction}\label{sec:intro}
\IEEEPARstart{T}{he} broad objective of distributed adaptation and learning is the solution of global, stochastic optimization problems by networked agents through localized interactions and in the absence of information about the statistical properties of the data.
When constant, rather than diminishing, step-sizes are employed, the resulting algorithms are adaptive in nature and are able to adapt to drifts in the data statistics. In this work, we consider a collection of \( N \) agents, where each agent \( k \) is equipped with a stochastic risk of the form \( J_k(w) = \E_x Q_k(w; \x_k) \) with \( Q_k(w;\x_k) \) referring to the loss function, \( w\in\mathds{R}^{M} \) denoting a parameter vector, and \( \x_k \) referring to the stochastic data. The expectation is over the probability distribution of the data. The objective of the network is to seek the Pareto solution:
\begin{equation}\label{eq:global_problem}
	\min_w J(w),\ \ \ \ \ \ \ \ \mathrm{where}\ J(w) \triangleq \sum_{k=1}^{N} p_k J_k(w)
\end{equation}
where the \( p_k \) are positive weights that are normalized to add up to one and will be specified further below; in particular, in the special case when the \(\{p_k\}\) are identical, they can be removed from~\eqref{eq:global_problem}. Algorithms for the solution of~\eqref{eq:global_problem} have been studied extensively over recent years both with inexact~\cite{Nedic09, Sayed14, Chen15transient, Xin19} and exact~\cite{Yuan16, Shi15, Yuan18} gradients. Here, we focus on the following diffusion strategy, which has been shown in previous works to provide enhanced performance and stability guarantees under constant step-size learning and adaptive scenarios~\cite{Sayed14proc, Sayed14}:
\begin{subequations}
\begin{align}
  \boldsymbol{\phi}_{k,i} &= \w_{k,i-1} - \mu \widehat{\nabla J}_{k}(\w_{k,i-1})\label{eq:adapt}\\
  \w_{k,i} &= \sum_{\ell=1}^{N} a_{\ell k} \boldsymbol{\phi}_{\ell,i}\label{eq:combine}
\end{align}
\end{subequations}
where \( \widehat{\nabla J}_{k}(\cdot) \) denotes a stochastic approximation for the true local gradient \( {\nabla J}_{k}(\cdot) \). {The intermediate estimate \( \boldsymbol{\phi}_{k,i} \) is obtained at agent \( k \) by taking a stochastic gradient update relative to the local cost \( J_k(\cdot) \). The intermediate estimates are then fused across local neighborhoods where} \( a_{\ell k} \) are convex combination weights satisfying:
\begin{equation}\label{eq:combinationcoef}
  a_{\ell k} \geq 0, \quad \sum_{\ell \in \mathcal{N}_k} a_{\ell k}=1, \quad a_{\ell k} = 0\ \mathrm{if}\ \ell \notin \mathcal{N}_k
\end{equation}
The symbol \( {\cal N}_k \) denotes the set of neighbors of agent \( k \).
\begin{assumption}[\textbf{Strongly-connected graph}]\label{as:strongly_connected}
	We shall assume that the graph described by the weighted combination matrix \(A=[a_{\ell k}]\) is strongly-connected~\cite{Sayed14}. This means that there exists a path with nonzero weights between any two agents in the network and, moreover, at least one agent has a nontrivial self-loop, \(a_{kk}>0\).\hfill\IEEEQED%
\end{assumption}

It then follows from the Perron-Frobenius theorem~\cite{Horn03,Pillai05,Sayed14} that \( A \) has a single eigenvalue at one while all other eigenvalues are strictly inside the unit circle, so that \( \rho(A)=1 \).  Moreover, if we let \(p\) denote the right-eigenvector of \(A\) that is associated with the eigenvalue at one, and if we normalize the entries of \(p\) to add up to one, then it also holds that all entries of \(p\) are strictly positive, i.e.,
\begin{equation}\label{eq:perron}
  Ap=p, \quad \mathds{1}^{\T} p=1, \quad p_k > 0
\end{equation}
where the \( \{ p_k \} \) denote the individual entries of the Perron vector, \(p\).

\subsection{Related Works}
The performance of the diffusion algorithm~\eqref{eq:adapt}--\eqref{eq:combine} has been studied extensively in differentiable settings~\cite{Sayed14proc, Chen15transient}, with extensions to multi-task~\cite{Nassif16}, constrained~\cite{Towfic15}, and non-differentiable~\cite{Ying18} environments. A common assumption in these works, along with others studying the behavior of distributed optimization algorithms in general, is that of \emph{convexity} (or strong-convexity) of the aggregate risk \(J(w)\).
While many problems of interest such as least-squares estimation~\cite{Sayed14}, logistic regression~\cite{Sayed14}, and support vector machines~\cite{Haerst98} are convex, there has been increased interest in the optimization of \emph{non-convex} cost functions. Such problems appear frequently in the design of robust estimators~\cite{Zoubir18} and the training of more complex machine learning architectures such as  those involving dictionary learning~\cite{Tosic11} and artificial neural networks~\cite{Choromanska14}.

Motivated by these applications, recent works have pursued the study of optimization algorithms for non-convex problems, both in the centralized and distributed settings~\cite{Gelfand91, Ge15, Lee16, Jin17, HadiDaneshmand18, Fang18, Allen18natasha, Allen18neon, Fang19, Jin19, Lorenzo16, Wang18, Tatarenko17, Scutari18, Swenson19, Jain17, Reddi16, Ge19}. While some works focus on establishing convergence to a stationary point~\cite{Lorenzo16, Wang18}, there has been growing interest in examining the ability of gradient descent implementations to escape from saddle points, since such points represent bottlenecks to the underlying learning problem~\cite{Choromanska14}. We defer a detailed discussion on the plethora of related works on second-order guarantees~\cite{Nesterov06, Gelfand91, Ge15, Lee16, Jin17, Scutari18, HadiDaneshmand18, Fang18, Allen18neon, Allen18natasha, Fang19, Jin19, Swenson19} to Part II~\cite{Vlaski19nonconvexP2}, where we will be establishing the ability of the diffusion strategy~\eqref{eq:adapt}--\eqref{eq:combine} to escape strict-saddle points efficiently. For ease of reference, the modeling conditions and results from this and related works are summarized in Table~\ref{tab:references}.

{The key contributions of Parts I and II this work are three-fold. To the best of our knowledge, we present the first analysis establishing \emph{efficient} (i.e., polynomial) escape from strict-staddle points in the \emph{distributed} setting. Second, we establish that the gradient noise process is sufficient to ensure efficient escape without the need to alter it by adding artificial forms of perturbations{, interlacing steps with small and large step-sizes, or imposing a dispersive noise assumption.} Third, relative to the existing literature on \emph{centralized} non-convex optimization, where the focus is mostly on deterministic or \emph{finite-sum} optimization, our modeling conditions are specifically tailored to the scenario of learning from stochastic \emph{streaming} data. In particular, we only impose bounds on the gradient noise variance in expectation, rather than assume a bound with probability one~\cite{HadiDaneshmand18, Fang19} or a sub-Gaussian distribution~\cite{Jin19}. Furthermore, we assume that any Lipschitz conditions only hold on the \emph{expected} stochastic gradient approximation, rather than for every realization, with probability one~\cite{Fang18, Allen18neon, Allen18natasha}.}
\begin{table*}\centering
\ra{1.3}
\begin{tabular}{@{}ccccccccc@{}}\toprule
& \multicolumn{5}{c}{Modeling conditions} & \phantom{abc}& \multicolumn{2}{c}{Results}\\ \cmidrule{2-6} \cmidrule{8-9}
& Gradient & Hessian & Initialization & Perturbations & Step-size && Stationary & Saddle \\ \midrule
\textbf{Centralized}\\
~\cite{Gelfand91}  & Lipschitz & --- & --- & SGD + Annealing & diminishing && \( \checkmark \) & asymptotic\(^{\dagger}\) \\
~\cite{Ge15}  & Lipschitz \& bounded\(^{\star}\) & Lipschitz & --- & i.i.d.\ and bounded w.p. 1 & constant && \( \checkmark \) & polynomial \\
~\cite{Lee16}  & Lipschitz & --- & Random & --- & constant && \( \checkmark \) & asymptotic \\
~\cite{Jin17}  & Lipschitz & Lipschitz & --- & Selective \& bounded w.p. 1 & constant && \( \checkmark \) & polynomial \\
~\cite{HadiDaneshmand18}  & Lipschitz & Lipschitz & --- & SGD, bounded w.p. 1 & alternating && \( \checkmark \) & polynomial \\
~\cite{Fang18}  & Lipschitz & Lipschitz & --- & Bounded variance, Lipschitz w.p. 1 & constant && \( \checkmark \) & polynomial \\
~\cite{Allen18natasha}  & Lipschitz & Lipschitz & --- & Bounded variance, Lipschitz w.p. 1 & constant && \( \checkmark \) & polynomial \\
~\cite{Allen18neon}  & Lipschitz & Lipschitz & --- & Bounded variance, Lipschitz w.p. 1 & constant && \( \checkmark \) & polynomial \\
~\cite{Fang19}  & Lipschitz & Lipschitz & --- & SGD, bounded w.p. 1 & constant && \( \checkmark \) & polynomial \\
~\cite{Jin19}  & Lipschitz & Lipschitz & --- & SGD + Gaussian &constant && \( \checkmark \)  & polynomial \\
\textbf{Decentralized}\\
~\cite{Lorenzo16} & Lipschitz \& bounded & --- & --- & ---&constant && \( \checkmark \) & ---\\
~\cite{Wang18} & Lipschitz & --- & --- & ---&constant && \( \checkmark \) & ---\\
~\cite{Tatarenko17} & Lipschitz \& bounded & --- & --- & i.i.d.&diminishing && \( \checkmark \) & ---\\
~\cite{Scutari18} & Lipschitz & Exists & Random & --- &constant && \( \checkmark \) & asymptotic\\
~\cite{Swenson19}  & Bounded disagreement & --- & --- & SGD + Annealing & diminishing && \( \checkmark \) & asymptotic\(^{\dagger}\) \\
\textbf{This work} & \textbf{Bounded disagreement} & \textbf{Lipschitz} & \textbf{---} & \textbf{Bounded moments} & \textbf{constant} && \( \boldsymbol{\checkmark} \) & \textbf{polynomial} \\ \bottomrule\
\end{tabular}
\caption{Comparison of modeling assumptions and results for gradient-based methods. Statements marked with \(^{\star}\) are not explicitly stated but are implied by other conditions. The works marked with \(^{\dagger}\) establish global (asymptotic) convergence, which of course implies escape from saddle-points.}\label{tab:references}
\end{table*}

\subsection{Preview of Results}
We first establish that in non-convex environments, as was already shown earlier in~\cite{Chen15transient} for convex environments, the evolution of the individual iterates \( \w_{k, i} \) at the agents continues to be well-described by the evolution of the weighted centroid vector \( \sum_{k=1}^N p_k \w_{k, i} \) in the sense that the iterates from across the network will cluster around this centroid after sufficient iterations. We subsequently consider two cases separately and establish descent in both of them. The first case corresponds to the region where the gradient at the network centroid is large and establish that descent can occur in one iteration. The second and more challenging case occurs when the gradient norm is small, but there is a sufficiently negative eigenvalue in the Hessian matrix. We establish Part II~\cite{Vlaski19nonconvexP2} that the recursion will continue to descend along the aggregate cost at a rate of \( O(\mu) \) per \( O(1/\mu) \) iterations. Combined with the first result, this descent relation allows us to provide guarantees about the second-order optimality of the returned iterates.

{The flow of the argument is summarized in Fig.~\ref{fig:classification_points}. We decompose \( \mathds{R}^M \) into the set of approximate first-order stationary points, i.e., those with \( \|\nabla J(w)\|^2 \le O(\mu) \) and the complement, i.e., the large-gradient regime. For the large-gradient regime, descent is established in Theorem~\ref{TH:DESCENT_RELATION}. We proceed to further decompose the set of approximate first-order stationary points into those that are \( \tau \)-strict-saddle, i.e., those that have a Hessian with significant negative eigenvalue \( \lambda_{\min}\left( \nabla^2 J(w) \right) \le -\tau \), and the complement, which are approximately second-order stationary points. For \( \tau \)-strict-saddle points we establish descent in Part II~\cite[Theorem 1]{Vlaski19nonconvexP2}. Finally, in Part II~\cite[Theorem 2]{Vlaski19nonconvexP2}, we conclude that the centroid will reach an approximately second-order stationary point in a \emph{polynomial} number of iterations.}
\begin{figure*}
	\centering
  \begin{tikzpicture}[grow cyclic, align=flush center,
                      level 1/.style={level distance=3.5cm, sibling angle = 40},
                      level 2/.style={level distance=4cm, sibling angle = 30}]
                      level 3/.style={level distance=7cm, sibling angle = 30}]
    \node[rounded corners, draw=blue!60, fill=blue!20, thick]{Network centroid \\\( \w_{c, i} \) at time \( i \)} [counterclockwise from=-20]
    	child {
      node[rounded corners, draw=green!60, fill=green!20, thick] {\textbf{NOT} \( O(\mu) \)-stationary \\  \( \|\nabla J(\w_{c, i})\|^2 > O(\mu) \)} [counterclockwise from=-15]
        child {
        node[rounded corners, draw=green!60, thick, fill=green!20] {Descent in one iteration by Theorem~\ref{TH:DESCENT_RELATION}: \\  \( \E \left \{ J(\w_{c, i}) - J(\w_{c, i+1}) | \w_{c, i} \in \mathcal{G} \right \} \ge O(\mu^2) \)}
        }
      }
    	child {
      node[rounded corners, draw=blue!60, fill=blue!20, thick] {\( O(\mu) \)-stationary \\  \( \|\nabla J(\w_{c, i})\|^2 \le O(\mu) \)} [counterclockwise from=-15]
        child {
          node[rounded corners, draw=red!60, thick, fill=red!20] {\( \tau \)-strict-saddle}
          child {
            node[rounded corners, draw=red!60, fill=red!20, thick]{Descent in \( i^s = O(1/(\mu \tau)) \) iterations in Part II~\cite[Theorem 1]{Vlaski19nonconvexP2}: \\  \( \E \left \{ J(\w_{c, i}) - J(\w_{c, i+i^s}) | \w_{c, i} \in \mathcal{H} \right \} \ge O(\mu) \)}
          }
        }
        child {
          node[rounded corners, draw=blue!60, thick, fill=blue!20] { \( \lambda_{\min}\left( \nabla^2 J(\w_{c, i}  \right) > -\tau  \)  }
          child {
            node[rounded corners, draw=blue!60, fill=blue!20, thick]{\( \w_{c, i} \) is approximately second-order stationary.}
          }
        }
      }
  ;
  \end{tikzpicture}
	\caption{Classification of approximately stationary points. Theorem~\ref{TH:DESCENT_RELATION} in this work establishes descent in the green branch. The red branch is treated in Part II~\cite[Theorem 1]{Vlaski19nonconvexP2}. The two results are combined in~\cite[Theorem 2]{Vlaski19nonconvexP2} to establish the return of a second-order stationary point with high probability.}\label{fig:classification_points}
\end{figure*}
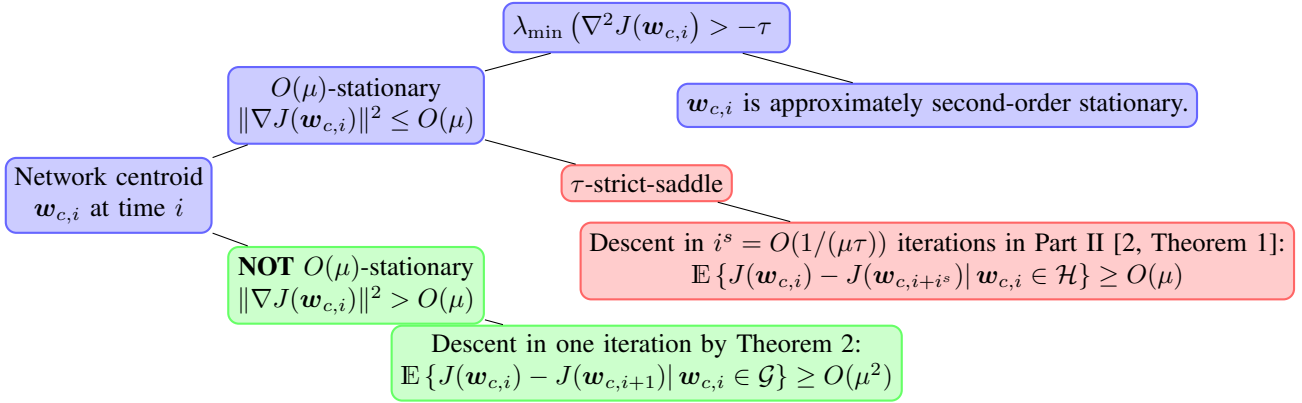

\section{Evolution Analysis}
We shall perform the analysis under the following common assumptions on the gradients and their approximations.
\begin{assumption}[\textbf{Lipschitz gradients}]\label{as:lipschitz}
  For each \( k \), the gradient \( \nabla J_k(\cdot) \) is Lipschitz, namely, for any \( x,y \in \mathds{R}^{M} \):
  \begin{equation}\label{eq:lipschitz}
    \|\nabla J_k(x) - \nabla J_k(y)\| \le \delta \|x-y\|
  \end{equation}
	In light of~\eqref{eq:global_problem} and Jensen's inequality, this implies for the aggregate cost:
	\begin{equation}\label{eq:lipschitz_global}
		\|\nabla J(x) - \nabla J(y)\| \le \delta \|x-y\|
	\end{equation}
\end{assumption}\hfill\IEEEQED%

{The Lipschitz gradient conditions~\eqref{eq:lipschitz} and~\eqref{eq:lipschitz_global} imply bounds on the both the function value and the Hessian matrix (when it exists), which will be used regularly throughout the derivations. In particular, we have for the function values:
\begin{align}
  J(y) \le J(x) + {\nabla J(x)}^{\T} \left( y-x \right)  + \frac{\delta}{2} {\|x-y\|}^2 \label{eq:quadratic_upper}
\end{align}
For the Hessian matrix we have~\cite{Sayed14}:
\begin{align}\label{eq:hessian_bound}
  - \delta I \le \nabla^2 J(x) \le \delta I
\end{align}
}
{
\begin{assumption}[\textbf{Bounded gradient disagreement}]\label{as:bounded}
  For each pair of agents \( k \) and \( \ell \), the gradient disagreement is bounded, namely, for any \( x \in \mathds{R}^{M} \):
  \begin{equation}\label{eq:bounded}
    \|\nabla J_k(x) - \nabla J_{\ell}(x)\| \le G
  \end{equation}
\end{assumption}\hfill\IEEEQED}

{This assumption is similar to the one used in~\cite{Swenson19} under a diminishing step-size with annealing. Note that condition~\eqref{eq:bounded} is weaker than the more common assumption of bounded gradients. Condition~\eqref{eq:bounded} is automatically satisfied in cases where the expected risks \( J_k(\cdot) \) are common (though agents still may see different realizations of data), or in the case of centralized stochastic gradient descent where the number of agents is one. This condition is also satisfied whenever agent-specific risks with bounded gradients are regularized by common regularizers with potentially unbounded gradients, as is common in many machine learning applications. Observe that~\eqref{eq:bounded} implies a similar condition on the deviation from the centralized gradient via Jensen's inequality:
\begin{align}
  &\: \left \|\nabla J_k(x) - \nabla J(x) \right \| =\left \| \sum_{\ell=1}^N p_{\ell}\left( \nabla J_k(x) - \nabla J_{\ell}(x) \right)\right\| \notag \\
  \le&\: \sum_{\ell=1}^N p_{\ell} \left\| \nabla J_k(x) - \nabla J_{\ell}(x)\right)\| \le G
\end{align}
}
\begin{definition}[\textbf{Filtration}]\label{def:filtration}
  We denote by \( \boldsymbol{\mathcal{F}}_{i} \) the filtration generated by the random processes \( \w_{k, j} \) for all \( k \) and \( j \le i \):
  \begin{equation}
    \boldsymbol{\mathcal{F}}_{i} \triangleq \left \{ \bcw_{0}, \bcw_{1}, \ldots, \bcw_{i} \right \}
  \end{equation}
  where \( \bcw_{j} \triangleq \mathrm{col}\left \{ \w_{1, j}, \ldots, \w_{k, j} \right \} \) contains the iterates across the network at time \( j \). Informally, \( \boldsymbol{\mathcal{F}}_{i} \) captures all information that is available about the stochastic processes \( \w_{k, j} \) across the network up to time \( i \).
\end{definition}\hfill\IEEEQED

{Throughout the following derivations, we will frequently rely on appropriate conditionings to make the analysis tractable. A frequent theme will be the exchange of conditioning on filtrations by conditioning on events. To this end, the following lemma will be used repeatedly.
\begin{lemma}[\textbf{Conditioning}]\label{LEM:CONDITIONING}
  Suppose \( \w \in \mathds{R}^M \) is a random variable measurable by \( \boldsymbol{\mathcal{F}} \). In other words, \( \w \) is deterministic conditioned on \( \boldsymbol{\mathcal{F}} \) and
  \begin{equation}
    \E \left \{ \w | \boldsymbol{\mathcal{F}} \right \} = \w
  \end{equation}
  Then,
  \begin{equation}
    \E \left \{ \E \left \{  \x | \boldsymbol{\mathcal{F}} \right \} | \w \in \mathcal{S} \right \} = \E \left \{ \x | \w \in \mathcal{S} \right \}
  \end{equation}
  for any deterministic set \( \mathcal{S} \subseteq \mathds{R}^M \) and random \( \x \in \mathds{R}^M \).
\end{lemma}
\begin{proof}
  Denote by \( \mathds{I}_{\mathcal{S}}(\w) \) the random indicator function:
  \begin{equation}
    \mathds{I}_{\mathcal{S}}(\w) = \begin{cases} 1, \ \mathrm{if} \ \w \in \mathcal{S} \\ 0, \ \mathrm{otherwise}. \end{cases}
  \end{equation}
  Since \( \w \) is measurable by \( \boldsymbol{\mathcal{F}} \), then \( \mathds{I}_{\mathcal{S}}(\w) \) is measurable by \( \boldsymbol{\mathcal{F}} \) as well. In other words, the event \( \w \in \mathcal{S} \) is deterministic conditioned on \( \boldsymbol{\mathcal{F}} \). Furthermore, for the random variable \( \x \mathds{I}_{\mathcal{S}}(\w) \), we have:
  \begin{align}
    \E \left \{  \x \mathds{I}_{\mathcal{S}}(\w)  \right \} =&\: \E \left \{ \x \mathds{I}_{\mathcal{S}}(\w) | \w \in \mathcal{S} \right \} \cdot \mathrm{Pr} \left \{ \w \in \mathcal{S} \right \} \notag \\
    &\: + \E \left \{ \x \mathds{I}_{\mathcal{S}}(\w) | \w \notin \mathcal{S} \right \} \cdot \mathrm{Pr} \left \{ \w \notin \mathcal{S} \right \} \notag \\
    =&\: \E \left \{ \x | \w \in \mathcal{S} \right \} \cdot \mathrm{Pr} \left \{ \w \in \mathcal{S} \right \}\label{eq:equiv_exp0}
  \end{align}
  Rearranging yields:
  \begin{equation}\label{eq:equiv_exp}
    \E \left \{  \x | \w \in \mathcal{S} \right \} = \frac{\E \left \{  \x \mathds{I}_{\mathcal{S}}(\w)  \right \}}{ \mathrm{Pr} \left \{ \w \in \mathcal{S}  \right \}}
  \end{equation}
  {Similarly, for the random variable \( \E \left \{ \x | \boldsymbol{\mathcal{F}} \right \} \mathds{I}_{\mathcal{S}}(\w) \), we have:
  \ifarx \begin{align}
    &\: \E \left \{  \E \left \{ \x | \boldsymbol{\mathcal{F}} \right \} \mathds{I}_{\mathcal{S}}(\w)  \right \}\notag \\
    =&\: \E \left \{ \E \left \{ \x | \boldsymbol{\mathcal{F}} \right \} \mathds{I}_{\mathcal{S}}(\w) | \w \in \mathcal{S} \right \} \cdot \mathrm{Pr} \left \{ \w \in \mathcal{S} \right \} \notag \\
    &\: + \E \left \{ \E \left \{ \x | \boldsymbol{\mathcal{F}} \right \} \mathds{I}_{\mathcal{S}}(\w) | \w \notin \mathcal{S} \right \} \cdot \mathrm{Pr} \left \{ \w \notin \mathcal{S} \right \} \notag \\
    =&\: \E \left \{ \E \left \{ \x | \boldsymbol{\mathcal{F}} \right \} | \w \in \mathcal{S} \right \} \cdot \mathrm{Pr} \left \{ \w \in \mathcal{S} \right \}\label{eq:equiv_exp2}
  \end{align}}
  \else\begin{align}
    \E \left \{  \E \left \{ \x | \boldsymbol{\mathcal{F}} \right \} \mathds{I}_{\mathcal{S}}(\w)  \right \} = \E \left \{ \E \left \{ \x | \boldsymbol{\mathcal{F}} \right \} | \w \in \mathcal{S} \right \} \cdot \mathrm{Pr} \left \{ \w \in \mathcal{S} \right \}\label{eq:equiv_exp2}
  \end{align}}\fi
  It then follows that:
  \begin{align}
    &\E \left \{ \E \left \{  \x | \boldsymbol{\mathcal{F}} \right \} | \w \in \mathcal{S} \right \} \stackrel{\eqref{eq:equiv_exp2}}{=}\: \frac{\E \left \{ \E \left \{  \x | \boldsymbol{\mathcal{F}} \right \} \mathds{I}_{\mathcal{S}}(\w)  \right \}}{ \mathrm{Pr} \left \{ \w \in \mathcal{S}  \right \}} \notag \\
    \stackrel{(a)}{=}&\: \frac{\E \left \{ \E \left \{  \x \mathds{I}_{\mathcal{S}}(\w) | \boldsymbol{\mathcal{F}} \right \}  \right \}}{\mathrm{Pr} \left \{ \w \in \mathcal{S}  \right \}} \stackrel{(b)}{=}\: \frac{\E \left \{  \x \mathds{I}_{\mathcal{S}}(\w)  \right \}}{\mathrm{Pr} \left \{ \w \in \mathcal{S}  \right \}} \notag \\
    \stackrel{\eqref{eq:equiv_exp}}{=}&\: \E \left \{  \x | \w \in \mathcal{S}  \right \}
  \end{align}
  where in step \( (a) \) we pulled \( \mathds{I}_{\mathcal{S}}(\w) \) into the inner expectation, since it is deterministic conditioned on \( \boldsymbol{\mathcal{F}} \) and \( (b) \) follows from the law of total expectation.
\end{proof}}
\begin{assumption}[\textbf{Gradient noise process}]\label{as:gradientnoise}
  For each \( k \), the gradient noise process is defined as
  \begin{equation}
    \s_{k,i}(\w_{k,i-1}) = \widehat{\nabla J}_k(\w_{k,i-1}) - \nabla J_k(\w_{k,i-1})
  \end{equation}
  and satisfies
  \begin{subequations}
    \begin{align}
      \E \left\{ \s_{k,i}(\w_{k,i-1}) | \boldsymbol{\mathcal{F}}_{i-1} \right\} &= 0 \label{eq:conditional_zero_mean}\\
      \E \left\{ \|\s_{k,i}(\w_{k,i-1})\|^4 | \boldsymbol{\mathcal{F}}_{i-1} \right\} &\le \sigma^4 \label{eq:gradientnoise_fourth}
    \end{align}
  \end{subequations}
  for some non-negative constant \( \sigma^4 \). We also assume that the gradient noise pocesses are pairwise uncorrelated over the space conditioned on \( \boldsymbol{\mathcal{F}}_{i-1} \), i.e.:
  \begin{equation}
    \E \left\{ \s_{k,i}(\w_{k,i-1}) \s_{\ell,i}(\w_{\ell,i-1})^{\T} |  \boldsymbol{\mathcal{F}}_{i-1} \right\} = 0\label{eq:uncorrelated_noise}
  \end{equation}
  \hfill\IEEEQED%
\end{assumption}
Property~\eqref{eq:conditional_zero_mean} means that the gradient noise construction is unbiased on average. Property~\eqref{eq:gradientnoise_fourth} means that the fourth-moment of the gradient  noise is bounded. These properties are automatically satisfied for several costs of interest~\cite{Sayed14, Sayed14proc}. Note, that the bound on the fourth-order moment, in light of Jensen's intequality, immediately implies:
\begin{align}
  &\: \E \left\{ \|\s_{k,i}(\w_{k,i-1})\|^2 | \boldsymbol{\mathcal{F}}_{i-1} \right\} \notag \\
  \ifarx =&\: \E \left\{ \sqrt{\|\s_{k,i}(\w_{k,i-1})\|^4} | \boldsymbol{\mathcal{F}}_{i-1} \right\} \notag \\ \fi
  \le&\: \sqrt{\E \left\{ \|\s_{k,i}(\w_{k,i-1})\|^4 | \boldsymbol{\mathcal{F}}_{i-1} \right\}} \stackrel{\eqref{eq:gradientnoise_fourth}}{\le} \sigma^2 \label{eq:gradientnoise}
\end{align}
While our primary interest is in the development of algorithms that allow for learning from \emph{streaming} data, we remark briefly that the results obtained in this work are equally applicable to empirical risk minimization via stochastic gradient descent, by assuming that the streaming data is selected according to a particular distribution.

\begin{example}[Empirical Risk Minimization]
  {\rm Suppose the costs \( J_k(\cdot) \) are empirical based on locally collected data \( {\left \{ x_{k, s} \right \}}_{s=1}^S \) and take the form:}
  \begin{equation}
    J_k(w) = \frac{1}{S} \sum_{s=1}^{S} Q(w, x_{k, s})
  \end{equation}
\end{example}
\noindent In empirical risk minimization (ERM) problems, we are interested in finding a vector \( w^o \) that minimizes the following empirical risk over the data across the \emph{entire} network:
\begin{equation}
  w^o \triangleq \arg\min_w \frac{1}{N} \sum_{k=1}^N \left( \frac{1}{S} \sum_{s=1}^S Q(w, x_{k, s}) \right)\label{eq:emp_global}
\end{equation}
If we introduce the uniformly-distributed random variable \( \x_{k} = x_{k, s} \) with probability \( \frac{1}{S} \) for all \( s \), then the cost~\eqref{eq:emp_global} is equivalent to solving:
\begin{equation}
  w^o = \arg\min_w \frac{1}{N} \sum_{k=1}^N \E_{\x_k} Q(w, \x_{k})
\end{equation}
which is of the same form as~\eqref{eq:global_problem} with \( p_k = \frac{1}{N} \). The resulting gradient noise process satisfies the assumptions imposed in this work under appropriate conditions on the risk \( Q(\cdot, \cdot) \). This observation has been leveraged to accurately quantify the performance of stochastic gradient descent, as well as mini-batch and importance sampling generalizations, for emprical minimization of convex risks in~\cite{Yuan16}.\hfill\IEEEQED%

\subsection{Network basis transformation}
\noindent In analyzing the dynamics of the distributed algorithm~\eqref{eq:adapt}--\eqref{eq:combine}, it is useful to introduce the following extended quantities by collecting variables from across the network:
\begin{align}
  \bcw_{i} &\triangleq \mathrm{col} \left \{ \w_{1,i}, \ldots, \w_{N,i} \right \} \\
  \mathcal{A} &\triangleq A \otimes I_M\\
  \widehat{\g}(\bcw_{i}) &\triangleq \mathrm{col} \left \{ \widehat{\nabla J_1}(\w_{1,i}), \ldots , \widehat{\nabla J}_N(\w_{N,i}) \right \}
\end{align}
where \( \otimes \) denotes the Kronecker product operation. We can then write the diffusion recursion~\eqref{eq:adapt}--\eqref{eq:combine} compactly as
\begin{equation}\label{eq:recursion}
	\bcw_i = \mathcal{A}^{\T} \left( \bcw_{i-1} - \mu \widehat{\g}(\bcw_{i-1}) \right)
\end{equation}
By construction, the combination matrix \( A \) is left-stochastic and primitive and hence admits a Jordan decomposition of
the form \( A = V_{\epsilon} J V_{\epsilon}^{-1} \) with~\cite{Sayed14,Chen15transient}:
\begin{equation}\label{eq:jordan}
  V_{\epsilon} = \left[ \begin{array}{cc} p & V_R \end{array} \right],
  \ \ J = \left[ \begin{array}{cc} 1 & 0\\0 & J_{\epsilon} \end{array}\right],
  \ \ V_{\epsilon}^{-1} = \left[ \begin{array}{c} \mathds{1}^{\T} \\ \vphantom{O^{O^{O^O}}} V_L^{\T} \end{array}\right]
\end{equation}
where \( J_{\epsilon} \) is a block Jordan matrix with the eigenvalues \( \lambda_2(A) \) through \( \lambda_N(A) \) on
the diagonal and \( \epsilon \) on the first lower sub-diagonal. The extended matrix \( \mathcal{A} \) then satisfies \(
\mathcal{A} = \mathcal{V}_{\epsilon} \mathcal{J} \mathcal{V}_{\epsilon}^{-1} \) with \( \mathcal{V}_{\epsilon} =
{V}_{\epsilon} \otimes I_N \), \( \mathcal{J} = J \otimes I_N \), \( \mathcal{V}_{\epsilon}^{-1} = {V}_{\epsilon}^{-1}
\otimes I_N \). The spectral properties of \( A \) and its corresponding eigendecomposition have been exploited extensively in the study of the diffusion learning strategy in the \emph{convex} setting~\cite{Sayed14, Chen15transient}, and will continue to be useful in \emph{non-convex} scenarios.

Multiplying both sides of~\eqref{eq:recursion} by \( \left( p^{\T} \otimes I \right) \) from the left, we obtain in light of~\eqref{eq:perron}:
\begin{equation}
	\left( p^{\T} \otimes I \right) \bcw_i = \left( p^{\T} \otimes I \right) \bcw_{i-1} - \mu \left( p^{\T} \otimes I \right) \widehat{\g}(\bcw_{i-1})
\end{equation}
Letting \( \w_{c, i} \triangleq \sum_{k=1}^K p_k \w_{k, i} = \left( p^{\T} \otimes I \right) \bcw_i \) and exploiting the block-structure of the gradient term, we find:
\begin{equation}
	\w_{c, i} = \w_{c, i-1} - \mu \sum_{k=1}^N p_k \widehat{\nabla J}_k (\w_{k, i-1})
\end{equation}
Note that \( \w_{c, i} \) is a convex combination of iterates across the network and can be viewed as a weighted centroid. The recursion for \( \w_{c, i} \) is reminiscent of a stochastic gradient step associated with the aggregate cost \( \sum_{k=1}^N p_k J_k(w) \) with the exact gradients \( \nabla J_k(\cdot) \) replaced by stochastic approximations \( \widehat{\nabla J}_k(\cdot) \) \emph{and} with the stochastic gradients evaluated at \( \w_{k, i-1} \), rather than \( \w_{c, i-1} \). In fact, we can write:
\begin{align}\label{eq:perturbed_gradient_descent}
	\w_{c, i} = \w_{c, i-1} - \mu \sum_{k=1}^N p_k {\nabla J}_k (\w_{c, i-1}) - \mu \boldsymbol{d}_{i-1} - \mu \s_i
\end{align}
where we defined the perturbation terms:
\begin{align}
	\boldsymbol{d}_{i-1} &\triangleq \sum_{k=1}^N p_k \left( {\nabla J}_k (\w_{k, i-1}) - {\nabla J}_k (\w_{c, i-1}) \right) \\
	\boldsymbol{s}_{i} &\triangleq \sum_{k=1}^N p_k \left( \widehat{\nabla J}_k (\w_{k, i-1}) - {\nabla J}_k (\w_{k, i-1}) \right) \label{eq:centralized_gradient_noise}
\end{align}
{We use the subscript \( i-1 \) for \( \boldsymbol{d}_{i-1} \) to emphasize that it depends on data up to time \( i-1 \), in contrast to \( \s_i \) which is also dependent on the most recent data from time \( i \).} Observe that \( \boldsymbol{d}_{i-1} \) arises from the disagreement within the network, and in particular that if each \( \w_{k, i-1} \) remains close to the network centroid \( \w_{c, i-1} \), this perturbation will be small in light of the Lipschitz condition~\eqref{eq:lipschitz} on the gradients. The second perturbation term \( \s_i \) arises from the noise introduced by stochastic gradient approximations at each agent. We now establish that recursion~\eqref{eq:perturbed_gradient_descent} will continue to exhibit some of the desired properties of (centralized) gradient descent, despite the presence of persistent and coupled perturbation terms.

\subsection{Network disagreement}
To begin with, we study more closely the evolution of the individual estimates \( \w_{k, i} \) relative to the network centroid \( \w_{c, i} \). Multiplying~\eqref{eq:recursion} by \( \mathcal{V}_R^{\T} \triangleq \left( {V}_R^{\T} \otimes I \right) \) from the left yields in light of~\eqref{eq:jordan}:
\begin{align}\label{eq:recursive_disagreement}
	\mathcal{V}_R^{\T} \bcw_i &= \mathcal{V}_R^{\T} \mathcal{A}^{\T} \bcw_{i-1} - \mu \mathcal{V}_R^{\T} \mathcal{A}^{\T}  \widehat{\g}(\bcw_{i-1}) \notag \\
	\ifarx &= \mathcal{V}_R^{\T} \mathcal{A}^{\T} \mathcal{V}_L \mathcal{V}_R^{\T} \bcw_{i-1} - \mu \mathcal{V}_R^{\T} \mathcal{A}^{\T} \mathcal{V}_L \mathcal{V}_R^{\T} \widehat{\g}(\bcw_{i-1}) \notag \\ \fi
	&= J_{\epsilon}^{\T} \mathcal{V}_R^{\T} \bcw_{i-1} - \mu J_{\epsilon}^{\T} \mathcal{V}_R^{\T} \widehat{\g}(\bcw_{i-1})
\end{align}
Then, for the deviation from the network centroid:
\begin{align}
  &\:\bcw_{i} - \bcw_{c, i} \notag \\
  \ifarx =&\:\bcw_{i} - \left( \mathds{1} p^{\T} \otimes I \right) \bcw_{i} \notag \\ \fi
  =&\: \left( I - \left( \mathds{1} p^{\T} \otimes I \right) \right) \bcw_{i} \notag \\
  =&\: \left( {\left({V_{\epsilon}}^{-1} \otimes I\right)}^{\T} {\left(V_{\epsilon} \otimes I\right)}^{\T} - \left( \mathds{1} p^{\T} \otimes I \right) \right) \bcw_{i} \notag \\
  \stackrel{\eqref{eq:jordan}}{=}&\: \mathcal{V}_L \mathcal{V}_R^{\T} \bcw_{i} \label{eq:deviation_from_centroid_transform}
\end{align}
so that the deviation from the centroid can be easily recovered from \( \mathcal{V}_R^{\T} \bcw_{i} \) in~\eqref{eq:recursive_disagreement}. Proceeding with~\eqref{eq:recursive_disagreement}, we find:
\begin{align}
	&\:{\left \| \mathcal{V}_R^{\T} \bcw_i \right \|}^4 \notag \\
	\ifarx =&\: {\left \| J_{\epsilon}^{\T} \mathcal{V}_R^{\T} \bcw_{i-1} - \mu J_{\epsilon}^{\T} \mathcal{V}_R^{\T} \widehat{\g}(\bcw_{i-1}) \right \|}^4 \notag \\ \fi
	\stackrel{(a)}{\le}&\: {\left \|J_{\epsilon}^{\T} \right \|}^4 {\left \| \mathcal{V}_R^{\T} \bcw_{i-1} - \mu \mathcal{V}_R^{\T} \widehat{\g}(\bcw_{i-1}) \right \|}^4 \notag \\
	\stackrel{(b)}{\le}&\: {\left \|J_{\epsilon}^{\T} \right \|} {\left \| \mathcal{V}_R^{\T} \bcw_{i-1}  \right \|}^4 + \mu^4 \frac{{\left \|J_{\epsilon}^{\T} \right \|}^4}{\left(1-{\left \|J_{\epsilon}^{\T} \right \|}\right)^3} {\left \| \mathcal{V}_R^{\T} \widehat{\boldsymbol{g}}(\bcw_{i-1}) \right \|}^4 \label{eq:disagreement_inequality}
\end{align}
where \( (a) \) follows from the sub-multiplicative property of norms, and \( (b) \) follows from Jensen's inequality \( {\|a + b\|}^4 \le \frac{1}{\alpha^3} \|a\|^4 + \frac{1}{\left(1-\alpha\right)^3} \|b\|^4 \) with
\begin{equation}
	\alpha = \left \|J_{\epsilon}^{\T} \right \| \triangleq \sqrt{\rho\left( J_{\epsilon} J_{\epsilon}^{\T} \right)} \le \sqrt{\left \| J_{\epsilon} J_{\epsilon}^{\T} \right \|_1} \le \sqrt{\lambda_2^2 + \epsilon^2} < 1
\end{equation}
for sufficiently small \( \epsilon \) due to Assumption~\ref{as:strongly_connected}, where \( \lambda_2 \triangleq \rho \left( A - \mathds{1} p^{\T} \right) \). We observe that the term \( {\left \| \mathcal{V}_R^{\T} \bcw_i \right \|}^4 \) contracts at an exponential rate given by \( \left \| J_{\epsilon}^{\T} \right \| \approx \lambda_2 \) for small \( \epsilon \), also known as the mixing rate of the graph. Iterating this relation and applying Assumptions~\ref{as:strongly_connected}--\ref{as:gradientnoise}, we obtain the following result.

\begin{theorem}[\textbf{Network disagreement (4th order)}]\label{LEM:NETWORK_DISAGREEMENT_FOURTH}
	Under assumptions~\ref{as:strongly_connected}--\ref{as:gradientnoise}, the network disagreement is bounded after sufficient iterations \( i \ge i_o \) by:
	\begin{align}
		&\:\E {\left \| \bcw_i - \left( \mathds{1} p^{\T} \otimes I \right) \bcw_{i} \right \|}^4 \notag \\
    \le&\: \mu^4 {\left \| \mathcal{V}_L \right \|}^4 \frac{{\left \|J_{\epsilon}^{\T} \right \|}^4}{{\left(1-{\left \|J_{\epsilon}^{\T} \right \|}\right)}^4} {\| \mathcal{V}_R^{\T} \|}^4 N^2 \left( G^4 + \sigma^4 \right) + o(\mu^4)\label{eq:network_disagreement_fourth}
	\end{align}
  where
  \begin{equation}
    i_o = \frac{\log\left( o(\mu^4) \right)}{\log\left( {\left \|J_{\epsilon}^{\T} \right \|} \right)}
  \end{equation}
  and \( o(\mu^4) \) denotes a term that is higher in order than \( \mu^4 \).
\end{theorem}
\begin{IEEEproof}
	Appendix~\ref{ap:network_disagreement_fourth}.
\end{IEEEproof}
\noindent Note again, that Jensen's inequality immediately implies for the second-order moment:
\begin{align}
  &\:\E {\left \| \bcw_i - \left( \mathds{1} p^{\T} \otimes I \right) \bcw_{i} \right \|}^2 \notag \\
  \ifarx =&\:\E \sqrt{{\left \| \bcw_i - \left( \mathds{1} p^{\T} \otimes I \right) \bcw_{i} \right \|}^4} \notag \\ \fi
  \le&\:\sqrt{\E {\left \| \bcw_i - \left( \mathds{1} p^{\T} \otimes I \right) \bcw_{i} \right \|}^4} \notag \\
  \stackrel{(a)}{\le}&\: \mu^2 {\left \| \mathcal{V}_L \right \|}^2 \frac{{\left \|J_{\epsilon}^{\T} \right \|}^2}{{\left(1-{\left \|J_{\epsilon}^{\T} \right \|}\right)}^2} {\| \mathcal{V}_R^{\T} \|}^2 N \left( G^2 + \sigma^2 \right) + o(\mu^2)\label{eq:network_disagreement}
\end{align}
where \( (a) \) follows from~\eqref{eq:network_disagreement_fourth} and sub-additivity of the square root, i.e. \( \sqrt{x+y} \le \sqrt{x} + \sqrt{y} \). This result establishes that, for every agent \( k \), we have after sufficient iterations \( i \ge i_o \):
\begin{equation}
  \E {\|\w_{k, i} - \w_{c, i} \|}^2 \le O(\mu^2)
\end{equation}
or, by Markov's inequality~\cite{Klenke13}:
\begin{equation}
  \mathrm{Pr} \left \{ {\|\w_{k, i} - \w_{c, i} \|}^2 \ge O(\mu) \right \} \le O(\mu)
\end{equation}
and hence \( \w_{k, i} \) will be arbitrarily close to \( \w_{c, i} \) with arbitrarily high probability for all agents. This result has two implications. First, it allows us to use the network centroid \( \w_{c, i} \) as a proxy for all iterates \( \w_{k, i} \) in the network, since all agents will cluster around the network centroid after sufficient iterations. Second, it allows us to bound the perturbation terms encountered in~\eqref{eq:perturbed_gradient_descent}.
\begin{lemma}[\textbf{Perturbation bounds (2nd and 4th order)}]\label{LEM:PERTURBATION_BOUNDS_FOURTH}
	Under assumptions~\ref{as:strongly_connected}--\ref{as:gradientnoise} and for sufficiently small step-sizes \( \mu \), the perturbation terms are bounded as:
	\begin{align}
		{\left( \E {\|\boldsymbol{d}_{i-1}\|}^2 \right)}^2 &\le \E {\|\boldsymbol{d}_{i-1}\|}^4 \le O(\mu^4) \label{eq:d_omufourth}\\
		{\left( \E \left \{ {\|\boldsymbol{s}_{i}\|}^2 | \boldsymbol{\mathcal{F}}_{i-1} \right \}\right)}^2 &\le \E \left \{ {\|\boldsymbol{s}_{i}\|}^4 | \boldsymbol{\mathcal{F}}_{i-1} \right \} \le \sigma^4
	\end{align}
	after sufficient iterations \( i \ge i_0 \).
\end{lemma}
\begin{IEEEproof}
	Appendix~\ref{ap:perturbation_bounds_fourth}.
\end{IEEEproof}

\begin{definition}[Sets]\label{DEF:SETS}
  To simplify the notation in the sequel, we introduce following sets:
  \begin{align}
    \mathcal{G} &\triangleq \left \{ w : {\left \| \nabla J(w) \right \|}^2 \ge \mu \frac{c_2}{c_1}\left(1+ \frac{1}{\pi}\right) \right \} \label{eq:define_g}\\
    \mathcal{G}^C &\triangleq \left \{ w : {\left \| \nabla J(w) \right \|}^2 < \mu \frac{c_2}{c_1} \left(1+\frac{1}{\pi} \right)\right \} \\
    \mathcal{H} &\triangleq \left \{ w : w \in \mathcal{G}^C, \lambda_{\min}\left( \nabla^2 J(w) \right) \le -\tau \right \} \label{eq:define_h}\\
    \mathcal{M} &\triangleq \left \{ w : w \in \mathcal{G}^C, \lambda_{\min}\left( \nabla^2 J(w) \right) > -\tau \right \} \label{eq:define_m}
  \end{align}
  where \( \tau \) is a small positive parameterm, \( c_1 \) and \( c_2 \) are constants:
  \begin{align}
		c_1 &\triangleq \frac{1}{2} \left( 1 - 2 \mu \delta \right) = O(1) \label{eq:define_c1}\\
		c_2 &\triangleq \delta \sigma^2 / 2 = O(1) \label{eq:define_c2}
	\end{align}
  and \( 0 < \pi < 1 \) is a parameter to be chosen. Note that \( \mathcal{G}^C = \mathcal{H} \cup \mathcal{M} \). We also define the probabilities \( \pi^{\mathcal{G}}_i \triangleq \mathrm{Pr}\left \{ \w_{c, i} \in \mathcal{G} \right \} \), \(\pi^{\mathcal{H}}_i \triangleq \mathrm{Pr}\left \{ \w_{c, i} \in \mathcal{H} \right \} \) and \(\pi^{\mathcal{M}}_i \triangleq \mathrm{Pr}\left \{ \w_{c, i} \in \mathcal{M} \right \}\). Then for all \( i \), we have \( \pi^{\mathcal{G}}_i + \pi^{\mathcal{H}}_i + \pi^{\mathcal{M}}_i = 1 \).\hfill\IEEEQED
\end{definition}
{The definitions~\eqref{eq:define_g}--\eqref{eq:define_m} decompose the parameter-space \( \mathds{R}^M\) into two disjoint sets \( \mathcal{G} \) and \( \mathcal{G}^C \), and further sub-divides \( \mathcal{G}^C \) into \( \mathcal{H} \) and \( \mathcal{M} \). The set \( \mathcal{G} \) denotes the set all points \( w \) where the norm of the gradient is large, while \( \mathcal{G}^C = \mathcal{H} \cup \mathcal{M} \) denotes the set of all points where the norm of the gradient is small, i.e., approximately first-order stationary points. In a manner similar to related works on the escape from strict-saddle points, we further decompose the set \( \mathcal{G}^C \) of approximate first-order stationary points into those points \( w \in \mathcal{H} \) that do have a significant negative eigenvalue, and those in \( \mathcal{M} \) that do not~\cite{Ge15, Jin17}. Points in the parameter space that have a small gradient norm and \emph{no} significant negative eigenvalue are referred to as \emph{second-order} stationary points, while points in \( \mathcal{H} \) are known as \emph{strict} saddle-points due to the presence of a strictly negative eigenvalue in the Hessian matrix.
In the sequel, we will establish descent for centroids in \( \mathcal{G} \) in Theorem~\ref{TH:DESCENT_RELATION} and centroids in \( \mathcal{H} \) in Part II~\cite[Theorem 1]{Vlaski19nonconvexP2}, and hence the approach of a point in \( \mathcal{M} \) with high probability after a polynomial number of iterations in Part II~\cite[Theorem 2]{Vlaski19nonconvexP2}. Second-order stationary points are generally more likely to be ``good'' minimizers than first-order stationary points, which could even correspond to local maxima. Furthermore, for a certain class of cost functions, known as ``strict-saddle'' functions, second-order stationary points always correspond to local minimia for sufficiently small \( \tau \)~\cite{Ge15}.

\subsection{Evolution of the network centroid}
Having established in~\eqref{eq:network_disagreement}, that after sufficient iterations, all agents in the network will have contracted around the centroid in a small cluster for small step-sizes, we can now leverage \( \w_{c, i} \) as a proxy for all \( \w_{k, i} \). From Assumption~\ref{as:lipschitz} and~\eqref{eq:quadratic_upper}, we have the following bound:
\begin{align}
	J(\w_{c, i}) \le&\: J(\w_{c, i-1}) + {\nabla J(\w_{c, i-1})}^{\T} \left( \w_{c, i} - \w_{c, i-1} \right) \notag \\
	&\:+ \frac{\delta}{2} {\left \| \w_{c, i} - \w_{c, i-1} \right \|}^2
\end{align}
From~\eqref{eq:perturbed_gradient_descent}, we then obtain:
\begin{align}\label{eq:lipschitz_before_exp}
	J(\w_{c, i}) \le&\: J(\w_{c, i-1}) - \mu {\left \| {\nabla J(\w_{c, i-1})} \right \|}^2 \notag \\
	&\:- \mu {\nabla J(\w_{c, i-1})}^{\T} \left( \boldsymbol{d}_{i-1} + \s_i \right) \notag \\
	&\:+ \mu^2 \frac{\delta}{2} {\left \| {\nabla J(\w_{c, i-1})} + \boldsymbol{d}_{i-1} + \s_i \right \|}^2
\end{align}
This relation, along with~\eqref{eq:perturbed_gradient_descent} and the results from Lemma~\ref{LEM:PERTURBATION_BOUNDS_FOURTH}, allow us to establish the following theorem.
\begin{theorem}[\textbf{Descent relation}]\label{TH:DESCENT_RELATION}
	Beginning at \( \w_{c, i-1} \) in the large gradient regime \( \mathcal{G} \), we can bound:
  \begin{align}\label{eq:descent_in_g}
    &\:\E \left \{ J(\w_{c, i}) | \w_{c, i-1} \in \mathcal{G} \right \} \notag \\
    \le&\: \E \left \{ J(\w_{c, i-1}) | \w_{c, i-1} \in \mathcal{G} \right \} - \mu^2 \frac{c_2}{\pi} + \frac{O(\mu^3)}{\pi_{i-1}^{\mathcal{G}}}
  \end{align}
	{as long as \( \pi_{i-1}^{\mathcal{G}} = \mathrm{Pr}\left \{ \w_{c, i-1} \in \mathcal{G} \right \} \neq 0 \)} where the relevant constants are listed in definition~\ref{DEF:SETS}. On the other hand, beginning at \( \w_{c, i-1} \in \mathcal{M} \), we can bound:
  \begin{align}\label{eq:ascent_bound}
    &\:\E \left \{ J(\w_{c, i}) | \w_{c, i-1} \in \mathcal{M} \right \} \notag \\
    \le&\: \E \left \{ J(\w_{c, i-1}) | \w_{c, i-1} \in \mathcal{M} \right \} + \mu^2 {c_2} + \frac{O(\mu^3)}{\pi_{i-1}^{\mathcal{M}}}
  \end{align}
  {as long as \( \pi_{i-1}^{\mathcal{M}} = \mathrm{Pr}\left \{ \w_{c, i-1} \in \mathcal{M} \right \} \neq 0 \)}.
\end{theorem}
\begin{IEEEproof}
	Appendix~\ref{ap:descent_relation}.
\end{IEEEproof}
{Relation~\eqref{eq:descent_in_g} guarantees a lower bound on the expected improvement when the gradient norm at the current iterate is suffiently large, i.e. \( \w_{c, i-1} \in \mathcal{G} \) is not an approximately first-order stationary point. On the other hand, when \( \w_{c, i-1} \in \mathcal{M} \), inequality~\eqref{eq:ascent_bound} it establishes an upper bound on the expected ascent. The respective bounds can be balanced by appropriately choosing \( \pi \), which will be leveraged in Part II~\cite{Vlaski19nonconvexP2}. We are left to treat the third possibility, namely \( \w_{c, i-1} \in \mathcal{H} \). In this case, since the norm of the gradient is small, it is no longer possible to guarantee descent in a single iteration. We shall study the dynamics in more detail in the sequel.}

\subsection{Behavior around stationary points}
In the vicinity of saddle-points, the norm of the gradient is not sufficiently large to guarantee descent at every iteration as indicated by~\eqref{eq:descent_in_g}. Instead, we will study the cumulative effect of the gradient, as well as perturbations, over several iterations. For this purpose, we introduce the following second-order condition on the cost functions, which is common in the literature~\cite{Sayed14, Ge15, Jin17}.
\begin{assumption}[\textbf{Lipschitz Hessians}]\label{as:lipschitz_hessians}
  Each \( J_k(\cdot) \) is twice-differentiable with Hessian \( \nabla^2 J_k(\cdot) \) and, there exists \( \rho \ge 0 \) such that:
  \begin{equation}
    {\| \nabla^2 J_k(x) - \nabla^2 J_k(y) \|} \le \rho \|x - y\|
  \end{equation}
  By Jensen's inequality, this implies that \( J(\cdot) = \sum_{k=1}^N p_k J_k(\cdot) \) also satisfies:
  \begin{equation}\label{eq:lipschitz_hessians}
    {\| \nabla^2 J(x) - \nabla^2 J(y) \|} \le \rho \|x - y\|
  \end{equation}\hfill\IEEEQED
\end{assumption}

{Let \( i^{\star} \) denote an arbitraty point in time. We use \( i^{\star} \) }{in order to emphasize approximately first-order stationary points, where the norm of the gradient is small. Such first-order stationary points \( \w_{c, i^{\star}} \in \mathcal{G}^{C} \) could either be in the set of second-order stationary points \( \mathcal{M} \) or in the set of strict-saddle points \( \mathcal{H} \). Our objective is to show that when \( \w_{c, i^{\star}} \in \mathcal{H} \), we can guarantee descent after several iterations. To this end,} starting at \( i^{\star} \), we have for \( i \ge 0 \):
\begin{equation}\label{eq:stationary_point}
  \w_{c, i^{\star}+i+1} = \w_{c, i^{\star}+i} - \mu {\nabla J} (\w_{c, i^{\star}+i}) - \mu \boldsymbol{d}_{i^{\star}+i} - \mu \s_{i^{\star}+i+1}
\end{equation}

Subsequent analysis will rely on an auxilliary model, referred to as a short-term model. It will be seen that this model is more tractable and evolves ``close'' to the true recursion under the second-order smoothness condition on the Hessian matrix~\eqref{eq:lipschitz_hessians} and as long as the iterates remain close to a stationary point. A similar approach has been introduced and used to great advantage in the form of a ``long-term model'' to derive accurate mean-square deviation performance expressions for strongly-convex costs in~\cite{Sayed14, Sayed14proc, Chen15performance, Zhao15}. The approach was also used to provide a ``quadratic approximation'' to establish the ability of stochastic gradient based algorithms to escape from strict saddle-points in the single-agent case under i.i.d. perturbations in~\cite{Ge15}.

For the driving gradient term in~\eqref{eq:stationary_point}, we have from the mean-value theorem~\cite{Sayed14}:
\begin{equation}
  {\nabla J} (\w_{c, i^{\star}+i}) - {\nabla J} (\w_{c, i^{\star}}) = \boldsymbol{H}_{i^{\star}+i} \left(\w_{c, i^{\star}+i} - \w_{c, i^{\star}} \right)
\end{equation}
where
\begin{equation}
  \boldsymbol{H}_{i^{\star}+i} \triangleq \int_0^1 \nabla^2 J\left( (1-t) \w_{c, i^{\star}+i} + t \w_{c, i^{\star}} \right) dt
\end{equation}
Subtracting~\eqref{eq:stationary_point} from \( \w_{c, i^{\star}} \), we obtain:
\begin{align}
  &\w_{c, i^{\star}} - \w_{c, i^{\star}+i+1} \notag \\
  =& \w_{c, i^{\star}} - \w_{c, i^{\star}+i} + \mu {\nabla J} (\w_{c, i^{\star}+i}) + \mu \boldsymbol{d}_{i^{\star}+i} + \mu \s_{i^{\star}+i+1} \notag \\
  =& \left( I - \mu \boldsymbol{H}_{i^{\star} + i} \right) \left( \w_{c, i^{\star}} - \w_{c, i^{\star}+i} \right) + \mu {\nabla J} (\w_{c, i^{\star}}) \notag \\
  &+ \mu \boldsymbol{d}_{i^{\star}+i} + \mu \s_{i^{\star}+i+1}
\end{align}
We introduce short-hand notation for the deviation:
\begin{equation}
  \widetilde{\w}_{i}^{i^{\star}} \triangleq \w_{c, i^{\star}} - \w_{c, i^{\star}+i}
\end{equation}
{Note that \( \widetilde{\w}_{i}^{i^{\star}} \) denotes the deviation of the network centroid \( \w_{c, i^{\star}+i} \) at time \( i^{\star}+i \) from the initial, approximately first-order stationary point \( \w_{c, i^{\star}} \). Establishing escape from saddle-points is equivalent to establishing the growth of \( \widetilde{\w}_{i}^{i^{\star}} \) whenever \( \w_{c, i^{\star}} \in \mathcal{H}\). We hence expect the deviation to grow over time, but would like to establish that \( \w_{c, i^{\star}+i}\) moves away from \( \w_{c, i^{\star}} \) in a direction of descent.} We can then write more compactly:
\begin{align}
  \widetilde{\w}_{i+1}^{i^{\star}} = &\: \left( I - \mu \boldsymbol{H}_{i^{\star} + i} \right) \widetilde{\w}_{i}^{i^{\star}} + \mu {\nabla J} (\w_{c, i^{\star}}) \notag \\
  &\: + \mu \boldsymbol{d}_{i^{\star}+i} + \mu \s_{i^{\star}+i+1} \label{eq:error_recursion}
\end{align}
The time-varying nature of \( \boldsymbol{H}_{i^{\star}+i} \) makes this recursion difficult to study. We hence introduce the following auxilliary recursion, initialized at \( \w_{c, i^{\star}}' = \w_{c, i^{\star}} \), where \( \boldsymbol{H}_{i^{\star}+i} \) is replaced by \( \nabla^2 J(\w_{c, i^{\star}}) \) and the perturbation term \( \mu \boldsymbol{d}_{i^{\star}+i} \) is omitted:
\begin{align}
  \w_{c, i^{\star}} - \w_{c, i^{\star}+i+1}' =&\: \left( I - \mu \nabla^2 J( \w_{c, i^{\star}}) \right) \left( \w_{c, i^{\star}} - \w_{c, i^{\star}+i}' \right) \notag \\
  &\: + \mu \nabla J(\w_{c, i^{\star}}) + \mu \s_{i^{\star}+i+1}
\end{align}
or, more compactly, with \( \widetilde{\w}_{i}'{}^{i^{\star}} \triangleq \w_{c, i^{\star}} - \w_{c, i^{\star}+i}' \)
\begin{align}
  \widetilde{\w}'{}^{i^{\star}}_{i+1} =&\: \left( I - \mu \nabla^2 J( \w_{c, i^{\star}}) \right) \widetilde{\w}_{i}'{}^{i^{\star}} + \mu \nabla J(\w_{c, i^{\star}}) + \mu \s_{i^{\star}+i+1} \label{eq:long_term_recursive}
\end{align}
Of course, this second model is only useful in studying the behavior of the original recursion~\eqref{eq:stationary_point} if the iterates generated by both models remain close to each other, which we shall prove to be true. \ifarx Specifically, if we write:
\begin{equation}
  {\w}_{i^{\star}+i+1}' = {\w}_{i^{\star}+i+1} + \boldsymbol{u}_{i^{\star}+i+1}
\end{equation}
then \( \boldsymbol{u}_{i^{\star}+i+1} \) will be shown to be negligible in some sense. \fi Results along this line have been established in the centralized and distributed contexts for strongly-convex costs~\cite{Sayed14, Sayed14proc} and in the centralized setting for strict saddle points~\cite{Ge15}. We show here that this conclusion holds more generally in the vicinity of \( O(\mu) \)-first-order stationary points. Before establishing deviation bounds, we establish a short lemma which will be used repeatedly.
\begin{lemma}[A limiting result]\label{LEM:LIMITING_RESULTS}
  For \( T, \mu, \delta > 0 \) and \( k \in \mathds{Z}_+ \) with \( \mu < \frac{1}{\delta} \), we have:
  \begin{align}
    \lim_{\mu \to 0} {\left( \frac{{(1+\mu \delta)}^k}{{\left(1-{\mu \delta}\right)}^{k-1}} \right)}^{\frac{T}{\mu}} = e^{-T \delta + 2k T \delta} = O(1)
  \end{align}
\end{lemma}
\begin{IEEEproof}
  Appendix~\ref{AP:LIMITING_RESULTS}.
\end{IEEEproof}

\begin{lemma}[Deviation bounds]\label{LEM:DEVIATION_BOUNDS}
  {Suppose \( \mathrm{Pr}\left\{ \w_{c, i^{\star}} \in \mathcal{H} \right \} \neq 0 \).} Then, the following quantities are conditionally bounded:
  \begin{align}
    \E \left \{ {\left \| \widetilde{\w}_{i}^{i^{\star}} \right \|}^2 | \w_{c, i^{\star}} \in \mathcal{H} \right \} &\le O(\mu) + \frac{O(\mu^2)}{\pi_{i^{\star}}^{\mathcal{H}}} \label{eq:ms_stability}\\
    \E \left \{ {\left \| \widetilde{\w}_{i}^{i^{\star}} \right \|}^3 | \w_{c, i^{\star}} \in \mathcal{H} \right \} &\le O(\mu^{3/2}) + \frac{O(\mu^3)}{{\pi_{i^{\star}}^{\mathcal{H}}}} \label{eq:mt_stability}\\
    \E \left \{ {\left \| \widetilde{\w}_{i}^{i^{\star}} \right \|}^4 | \w_{c, i^{\star}} \in \mathcal{H} \right \} &\le O(\mu^{2}) + \frac{O(\mu^{4})}{\pi_{i^{\star}}^{\mathcal{H}}} \label{eq:mf_stability}\\
    \E \left \{ {\left \| \widetilde{\w}_{i}^{i^{\star}} - \widetilde{\w}_{i}'{}^{i^{\star}} \right \|}^2 | \w_{c, i^{\star}} \in \mathcal{H} \right \} &\le O(\mu^{2}) + \frac{O(\mu^{2})}{\pi_{i^{\star}}^{\mathcal{H}}} \label{eq:model_deviation}\\
    \E \left \{ {\left \| \widetilde{\w}_{i}'{}^{i^{\star}} \right \|}^2 | \w_{c, i^{\star}} \in \mathcal{H} \right \} &\le O(\mu) + \frac{O(\mu^2)}{\pi_{i^{\star}}^{\mathcal{H}}} \label{eq:longterm_deviation}
  \end{align}
  for \( i \le \frac{T}{\mu} \), where \( T \) denotes an arbitrary constant that is independent of the step-size \( \mu \).
\end{lemma}
\begin{IEEEproof}
  Appendix~\ref{AP:DEVIATION_BOUNDS}.
\end{IEEEproof}
{These deviation bounds establish that, beginning at a strict-saddle point \( \w_{c, i^{\star}} \) at time \( i^{\star} \) the iterates will remain close to \( \w_{c, i^{\star}} \) for the next \( O(1/\mu) \) iterations. Consequently, the stort-term model will be sufficiently accurate for the next \( O(1/\mu) \) iterations. We will establish formally in Part II~\cite{Vlaski19nonconvexP2} that the small-deviation bounds in Lemma~\ref{LEM:DEVIATION_BOUNDS} ensure descent of the true recursion can be inferred by studying only the evolution of the short-term model, which is significantly more tractable.}

\section{Application: Robust Regression}
\noindent Consider a scenario where each agent \( k \) in the network observes streaming realizations \( \left \{ \boldsymbol{\gamma}(k, i), \boldsymbol{h}_{k,i} \right \} \) from the linear model \( \boldsymbol{\gamma}(k) = \boldsymbol{h}_k^{\mathsf{T}} w^o + \boldsymbol{v}(k) \) where \( \boldsymbol{\gamma}(k) \) denotes scalar observations and \( \boldsymbol{v}(k) \) denotes measurement noise. One common approach for estimating \( w^o \) in a distributed setting is via least-mean-square error estimation, resulting in the local cost functions:
\begin{equation}
	J_k^{\mathrm{MSE}}(w) = \E {\left \|\boldsymbol{\gamma}(k) -  \boldsymbol{h}_k^{\mathsf{T}} w \right \|}^2
\end{equation}
The resulting problem is convex and has been studied extensively in the literature. While effective under the assumption of Gaussian noise, and similar well-behaved noise conditions, this approach is susceptible to outliers caused by heavy-tailed distributions for \( \boldsymbol{v}(k) \)~\cite{Zoubir18}.
This is caused by the fact that the quadratic risk penalizes errors proportionally to their squared norm, and as such has a tendency to over-correct outliers, even if they are rare. Several alternative robust cost functions have been suggested in the literature. We consider two in particular in order to illustrate the advantages of allowing for non-convex costs in the context of robust estimation, namely the Huber loss \( Q_k^{\mathrm{H}}(w; \x_k) \) and Tukey's biweight loss \( Q_k^{\mathrm{B}}(w; \x_k) \)~\cite{Zoubir18}. For ease of notation, let \( \boldsymbol{e}(w) \triangleq \boldsymbol{\gamma}(k) -  \boldsymbol{h}_k^{\mathsf{T}} w \). Then:
\begin{align}
	Q_k^{\mathrm{H}}(w; \x_k) &=
	\begin{cases}
		\frac{1}{2}{|\boldsymbol{e}(w)|}^2, \ &\mathrm{for}\ |\boldsymbol{e}(w)| \le c_H\\
		{c_H |\boldsymbol{e}(w)|} - \frac{1}{2} c_H^2, \ &\mathrm{for}\ |\boldsymbol{e}(w)| > c_H.\\
  \end{cases} \\
	Q_k^{\mathrm{B}}(w; \x_k) &=
	\begin{cases}
		\frac{c_B^2}{6}\left( 1 - {\left( 1-\frac{{{|\boldsymbol{e}(w)|}^2}}{c_B^2} \right)}^3 \right), \ &\mathrm{for}\ |\boldsymbol{e}(w)| \le c_B\\
		\frac{c_B^2}{6} \ &\mathrm{otherwise}
  \end{cases}
\end{align}
where \( c_H, c_B \) are tuning constants. The Huber cost is merely convex (and not strongly-convex), while the Tukey loss is non-convex. Both losses satisfy assumptions~\ref{as:strongly_connected}--\ref{as:gradientnoise} imposed in this work. In particular, since the Huber risk \( J_k^{\mathrm{H}}(w) \) has a unique, local minimum, which also happens to be locally strongly-convex, we can conclude that despite the absence of strong-convexity, the algorithm will converge to within \( O(\mu) \) of the global minimum. The Tukey loss on the other hand, is non-convex, and is therefore a more challenging problem. The setting for the simulation results is shown in Figure~\ref{fig:setting}.
\begin{figure}[!t]
\centering
\subfloat{
  \includegraphics[width=.45\linewidth]{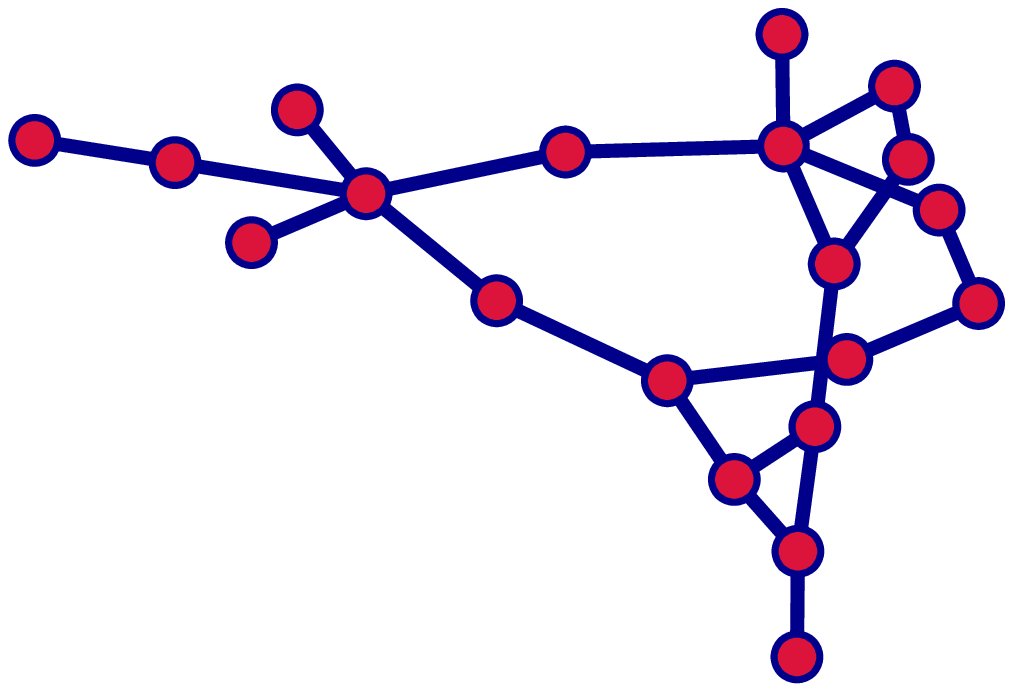}%
}
\hfil
\subfloat{
  \includegraphics[width=.45\linewidth]{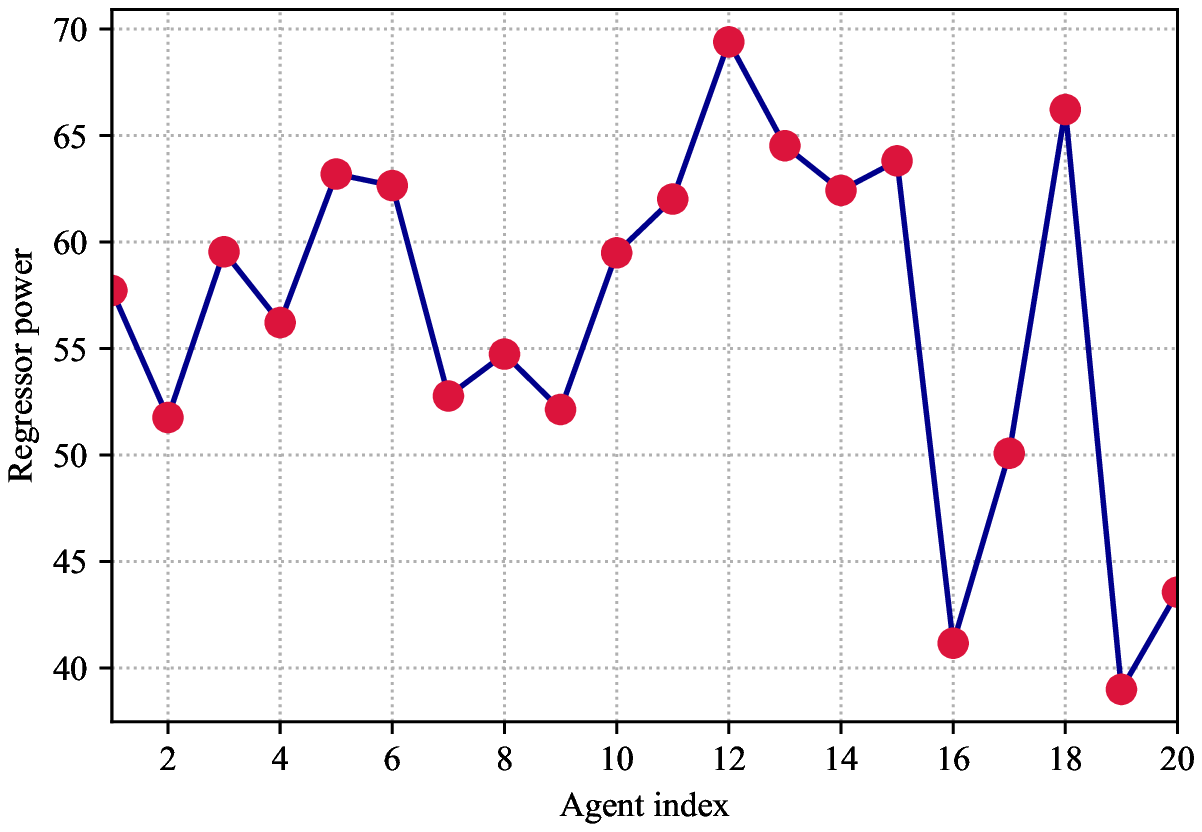}
}
\caption{Graph with \( N = 20 \) nodes (left) and regressor power \( \mathrm{Tr}\left( R_{h,k} \right) \) at each agent (right).}\label{fig:setting}
\end{figure}

Performance is illustrated in Fig.~\ref{fig:performance}. We first show the performance of each cost in the nominal scenario, where \( \boldsymbol{v}(k) \sim \mathcal{N}(0, \sigma_v^2) \). We observe that the distributed strategies outperform the non-cooperative ones, and that despite differences in the rate of convergence, there is negligible difference in the performance of the mean-square-error, Huber and Tukey variations. In the presence of outliers, modeled as a bimodal distribution with \( \boldsymbol{v}(k) \sim (1-\epsilon) \mathcal{N}(0, \sigma_v^2) + \epsilon \mathcal{N}( 10, \sigma_v^2 )\) and \( \epsilon = 0.1 \), the performance of the mean-square-error solution dramatically deteriorates, as is to be expected in the presence of deviations from the nominal model.
\begin{figure}[!t]
\centering
\subfloat{
  \includegraphics[width=.45\linewidth]{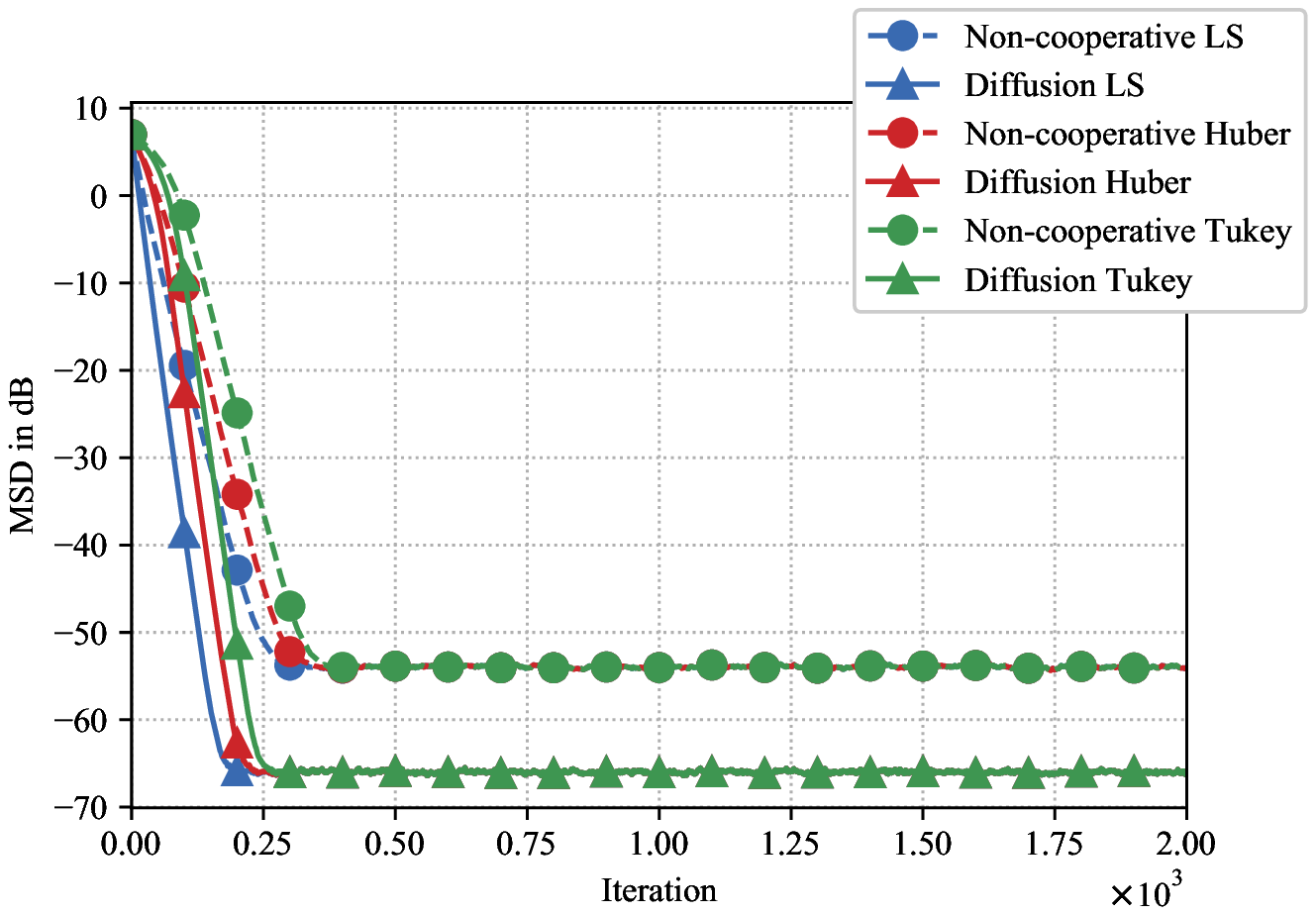}%
}
\hfil
\subfloat{
  \includegraphics[width=.45\linewidth]{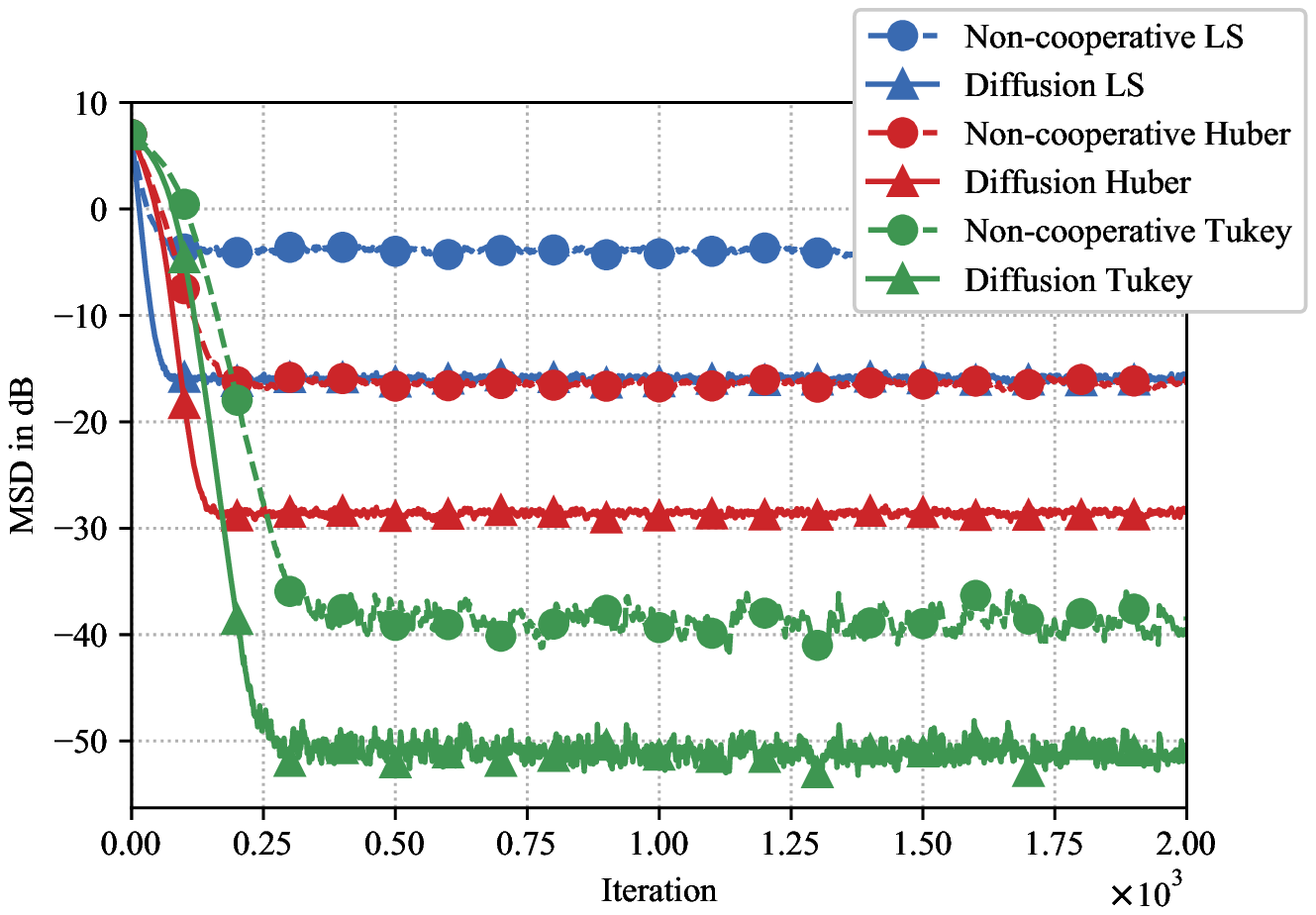}
}
\caption{Performance in the nominal (left) and corrupted case (right).}\label{fig:performance}
\end{figure}

\appendices%
\section{Proof of Lemma~\ref{LEM:NETWORK_DISAGREEMENT_FOURTH}}\label{ap:network_disagreement_fourth}
\noindent Starting from~\eqref{eq:recursive_disagreement}, taking norms of both sides and computing the fourth power, we find:
\begin{align}
  &\:{\left \| \mathcal{V}_R^{\T} \bcw_i \right \|}^4 \notag \\
  \ifarx=&\: {\left \| J_{\epsilon}^{\T}\mathcal{V}_R^{\T} \bcw_{i-1} + \mu  J_{\epsilon}^{\T}\mathcal{V}_R^{\T} \widehat{\boldsymbol{g}}(\bcw_{i-1}) \right \|}^4 \notag \\ \fi
  \le&\:  {\left\| J_{\epsilon}^{\T} \right\|}^4 {\left \|\mathcal{V}_R^{\T} \bcw_{i-1} + \mu  \mathcal{V}_R^{\T} \widehat{\boldsymbol{g}}(\bcw_{i-1}) \right \|}^4 \notag \\
  \stackrel{(a)}{\le}&\:  {\left\| J_{\epsilon}^{\T} \right\|} {\left \|\mathcal{V}_R^{\T} \bcw_{i-1} \right \|}^4 + \mu^4 \frac{{\left\| J_{\epsilon}^{\T} \right\|}^4}{{\left(1-{\left\| J_{\epsilon}^{\T} \right\|}\right)}^3} {\left\|  \mathcal{V}_R^{\T} \widehat{\boldsymbol{g}}(\bcw_{i-1}) \right \|}^4\label{eq:intermediate_dis}
\end{align}
where step \( (a) \) follows from convexity of \( \|\cdot\|^4 \) and Jensen's inequality, i.e. \( {\left \|a + b \right\|}^4 = \frac{1}{\alpha^3}{\left\| a\right\|}^4 + \frac{1}{{(1-\alpha)}^3}{\left\| b\right\|}^4 \). {To begin with, we study the stochastic gradient term in some greater detail. We have:
\begin{align}
  &\: {\left\|  \mathcal{V}_R^{\T} \widehat{\boldsymbol{g}}(\bcw_{i-1}) \right \|}^4 \notag \\
  =&\: {\left\|  \mathcal{V}_R^{\T} {{g}}(\bcw_{i-1}) + \mathcal{V}_R^{\T}\mathrm{col}\left\{\s_{k, i}(\w_{k, i-1})\right\} \right \|}^4 \notag \\
  \le&\: 8 {\left\|  \mathcal{V}_R^{\T} {{g}}(\bcw_{i-1})\right \|}^4 + 8{\left\|  \mathcal{V}_R^{\T}\mathrm{col}\left\{\s_{k, i}(\w_{k, i-1})\right\} \right \|}^4
\end{align}
For the first term we have:
\begin{align}
  &\: 8 {\left\|  \mathcal{V}_R^{\T} {{g}}(\bcw_{i-1})\right \|}^4 \notag \\
  \stackrel{(a)}{=}&\: 8 {\left\|  \mathcal{V}_R^{\T} {{g}}(\bcw_{i-1}) - \left( \mathds{1} p^{\T} \otimes I \right) {{g}}(\bcw_{i-1}) \right \|}^4 \notag \\
  \stackrel{(b)}{\le}&\: 8 {\left\|\mathcal{V}_R^{\T} \right\|}^4 {\left\|  {{g}}(\bcw_{i-1}) - \left( \mathds{1} p^{\T} \otimes I \right) {{g}}(\bcw_{i-1}) \right \|}^4 \notag \\
  \stackrel{(c)}{=}&\: 8 {\left\|\mathcal{V}_R^{\T} \right\|}^4 {\left( \sum_{k=1}^N {\left\| \nabla J_k(\w_{k, i-1}) - \nabla J(\w_{k, i-1}) \right \|}^2 \right)}^2 \notag \\
  \stackrel{\eqref{eq:bounded}}{\le}&\: 8 {\left\|\mathcal{V}_R^{\T} \right\|}^4 {\left( \sum_{k=1}^N G^2 \right)}^2 {\le}\:  8 {\left\|\mathcal{V}_R^{\T} \right\|}^4 N^2 G^4
\end{align}
where \( (a) \) follows from the fact that~\eqref{eq:jordan} implies \( V_R^{\T} \mathds{1} = 0 \), \( (b) \) follows from the sub-multiplicity of norms and \( (c) \) expands \( {\|\cdot\|}^2 \).} {For the gradient noise term we find under expectation:
\begin{align}
  &\: 8 \E {\left\|  \mathcal{V}_R^{\T}\mathrm{col}\left\{\s_{k, i}(\w_{k, i-1})\right\} \right \|}^4 \notag \\
  =&\: 8 {\left\|  \mathcal{V}_R^{\T}\right \|}^4 \E {\left\|  \mathrm{col}\left\{\s_{k, i}(\w_{k, i-1})\right\} \right \|}^4 \notag \\
  =&\: 8 {\left\|  \mathcal{V}_R^{\T}\right \|}^4 \E {\left( \sum_{k=1}^N {\left\|  \s_{k, i}(\w_{k, i-1}) \right \|}^2 \right)}^2 \notag \\
  \stackrel{(a)}{\le}&\: 8 {\left\|  \mathcal{V}_R^{\T}\right \|}^4 N \sum_{k=1}^N \E {\left\|  \s_{k, i}(\w_{k, i-1}) \right \|}^4  \notag \\
  \stackrel{\eqref{eq:gradientnoise_fourth}}{\le}&\: 8 {\left\|  \mathcal{V}_R^{\T}\right \|}^4 N \sum_{k=1}^N \sigma^4 =\: 8 {\left\|  \mathcal{V}_R^{\T}\right \|}^4 N^2  \sigma^4
\end{align}
where \( (a) \) follows from Cauchy-Schwarz, which implies \( {\left( \sum_{k=1}^N x_k \right)}^2 \le N \sum_{k=1}^N x_k^2  \). Plugging these relations back into~\eqref{eq:intermediate_dis}, we obtain:
\begin{align}
  &\:\E  {\left \| \mathcal{V}_R^{\T} \bcw_i \right \|}^4  \notag \\
  \le&\: {\left \|J_{\epsilon}^{\T} \right \|} \E {\left \| \mathcal{V}_R^{\T} \bcw_{i-1}  \right \|}^4 + \mu^4 \frac{{8}{\left \|J_{\epsilon}^{\T} \right \|}^4}{{\left(1-{\left \|J_{\epsilon}^{\T} \right \|}\right)}^3} {\| \mathcal{V}_R^{\T} \|}^4 N^2 \left( G^4 + \sigma^4 \right)
\end{align}
We can iterate, starting from \( i = 0 \), to obtain:
\begin{align}
  &\:\E  {\left \| \mathcal{V}_R^{\T} \bcw_i \right \|}^4  \notag \\
  \le&\: {\left \|J_{\epsilon}^{\T} \right \|}^i \E {\left \| \mathcal{V}_R^{\T} \cw_{0}  \right \|}^4 \notag \\
  &\:+ \mu^4 \frac{{8}{\left \|J_{\epsilon}^{\T} \right \|}^4}{{\left(1-{\left \|J_{\epsilon}^{\T} \right \|}\right)}^3} {\| \mathcal{V}_R^{\T} \|}^4 N^2 \left( G^4 + \sigma^4 \right) \sum_{n=1}^i \|J_{\epsilon}^{\T}\|^{n-1} \notag \\
  \stackrel{(a)}{\le}&\: {\left \|J_{\epsilon}^{\T} \right \|}^i \E {\left \| \mathcal{V}_R^{\T} \cw_{0}  \right \|}^4 \notag \\
  &\:+ \mu^4 \frac{{8}{\left \|J_{\epsilon}^{\T} \right \|}^4}{{\left(1-{\left \|J_{\epsilon}^{\T} \right \|}\right)}^4} {\| \mathcal{V}_R^{\T} \|}^4 N^2 \left( G^4 + \sigma^4 \right) \notag \\
  \stackrel{(b)}{\le}&\: o(\mu^4) + \mu^4 \frac{{8}{\left \|J_{\epsilon}^{\T} \right \|}^4}{{\left(1-{\left \|J_{\epsilon}^{\T} \right \|}\right)}^4} {\| \mathcal{V}_R^{\T} \|}^4 N^2 \left( G^4 + \sigma^4 \right) \label{eq:just_derived_fourth}
\end{align}
where \( (a) \) follows from \( \sum_{n=1}^i \|J_{\epsilon}^{\T}\|^{n-1} \le \sum_{n=1}^{\infty} \|J_{\epsilon}^{\T}\|^{n-1} = {\left( 1 - \|J_{\epsilon}^{\T}\| \right)}^{-1} \), and \( (b) \) holds whenever:
\begin{align}
  &\: {\left \|J_{\epsilon}^{\T} \right \|}^i \E {\left \| \mathcal{V}_R^{\T} \cw_{0}  \right \|}^4 \le o(\mu^4) \Longleftrightarrow\: {\left \|J_{\epsilon}^{\T} \right \|}^i \le o(\mu^4) \notag \\
  \Longleftrightarrow\: & i \log\left( {\left \|J_{\epsilon}^{\T} \right \|} \right) \le \log\left( o(\mu^4) \right) \Longleftrightarrow\: i \ge \frac{\log\left( o(\mu^4) \right)}{\log\left( {\left \|J_{\epsilon}^{\T} \right \|} \right)}\label{eq:intermediate_4444444}
\end{align}
Finally, we have from~\eqref{eq:deviation_from_centroid_transform} under~\eqref{eq:intermediate_4444444}:
\begin{align}
  &\:\E {\left \| \bcw_i - \left( \mathds{1} p^{\T} \otimes I \right) \bcw_i \right \|}^4 =\:\E {\left \| \mathcal{V}_L \mathcal{V}_R^{\T} \bcw_i \right \|}^4 \notag \\
  \ifarx \stackrel{(a)}{\le}&\: {\left \| \mathcal{V}_L \right \|}^4 \E {\left \| \mathcal{V}_R^{\T} \bcw_i \right \|}^4 \notag \\ \fi
  \stackrel{\eqref{eq:just_derived_fourth}}{\le}&\: \mu^4 {\left \| \mathcal{V}_L \right \|}^4 \frac{{\left \|J_{\epsilon}^{\T} \right \|}^4}{{\left(1-{\left \|J_{\epsilon}^{\T} \right \|}\right)}^4} {\| \mathcal{V}_R^{\T} \|}^4 N^2 \left( G^4 + \sigma^4 \right) + o(\mu^4)\label{eq:intermediate_13513523235}
\end{align}
\ifarx where \( (a) \) follows from the sub-multiplicative property of norms.\fi We conclude that all agents in the network will contract around the centroid vector \( \left( \mathds{1} p^{\T} \otimes I \right) \bcw_i  \) after sufficient iterations.}

\section{Proof of Lemma~\ref{LEM:PERTURBATION_BOUNDS_FOURTH}}\label{ap:perturbation_bounds_fourth}
\noindent We begin by studying the perturbation term \( \boldsymbol{s}_{i} \). We have:
\begin{align}
  &\: \E\left \{ {\| \boldsymbol{s}_{i} \|}^4 | \boldsymbol{\mathcal{F}}_{i-1} \right \} \notag \\
  =&\: \E  \left \{ { \left \| \sum_{k=1}^N p_k \left( \widehat{\nabla J}_k (\w_{k, i-1}) - {\nabla J}_k (\w_{k, i-1}) \right) \right \|}^4 | \boldsymbol{\mathcal{F}}_{i-1} \right \} \notag \\
  \stackrel{(a)}{\le}&\: \sum_{k=1}^N p_k \E \left \{ { \left \| \widehat{\nabla J}_k (\w_{k, i-1}) - {\nabla J}_k (\w_{k, i-1}) \right \|}^4 | \boldsymbol{\mathcal{F}}_{i-1} \right \} \notag \\
  \stackrel{(b)}{\le}&\: \sum_{k=1}^N p_k \sigma^4 =\: \sigma^4
\end{align}
where \( (a) \) follows from \( \sum_{k=1}^N p_k = 1 \) and Jensen's inequality and \( (b) \) follows from the fourth-order moment condition in Assumption~\ref{as:gradientnoise}. For the second perturbation term, we have
\begin{align}
  {\|\boldsymbol{d}_{i-1}\|}^4 &= {\left \| \sum_{k=1}^N p_k \left( {\nabla J}_k (\w_{k, i-1}) - {\nabla J}_k (\w_{c, i-1}) \right)\right \|}^4 \notag \\
  &\stackrel{(a)}{\le}  \sum_{k=1}^N p_k {\left \| {\nabla J}_k (\w_{k, i-1}) - {\nabla J}_k (\w_{c, i-1}) \right \|}^4 \notag \\
  &\stackrel{(b)}{\le}  \delta^4 \sum_{k=1}^N p_k {\left \| \w_{k, i-1} - \w_{c, i-1} \right \|}^4 \notag \\
  &\le  \delta^4 p_{\max} \sum_{k=1}^N {\left \| \w_{k, i-1} - \w_{c, i-1} \right \|}^4 \notag \\
  &\le  \delta^4 p_{\max} {\left( \sum_{k=1}^N {\left \| \w_{k, i-1} - \w_{c, i-1} \right \|}^2 \right)}^2 \notag \\
  &=  \delta^4 p_{\max}  {\left \| \bcw_{i-1} - \bcw_{c, i-1} \right \|}^4 \label{eq:d_bound_fourth}
\end{align}
where \( (a) \) again follows from Jensen's inequality, \( (b) \) follows from the Lipschitz gradient condition in Assumption~\ref{as:lipschitz}, and we introduced \( \bcw_{c, i-1} \triangleq \mathds{1} \otimes \w_{c, i-1} \). Result~\eqref{eq:d_omufourth} follows by applying~\eqref{eq:intermediate_13513523235} to~\eqref{eq:d_bound_fourth}.

\section{Proof of Lemma~\ref{LEM:LIMITING_RESULTS}}\label{AP:LIMITING_RESULTS}
\ifarx \noindent For the natural logarithm of the expression, we have:
\begin{align}
  &\: \log{\left( \frac{{(1+\mu \delta)}^k}{{\left(1-{\mu \delta}\right)}^{k-1}} \right)}^{\frac{T}{\mu}} \notag \\
  =&\: \frac{T}{\mu} \left( k\log\left(1+\mu \delta\right) - (k-1) \log {\left( 1-\mu \delta \right)} \right)
\end{align} \fi
Since the logarithm is continuous over \( \mathds{R}_{+} \), we have:
\begin{align}
  &\: \log \left( \lim_{\mu \to 0} {\left( \frac{{(1+\mu \delta)}^k}{{\left(1-{\mu \delta}\right)}^{k-1}} \right)}^{\frac{T}{\mu}} \right) \notag \\
  =&\: \lim_{\mu \to 0} \log \left( {\left( \frac{{(1+\mu \delta)}^k}{{\left(1-{\mu \delta}\right)}^{k-1}} \right)}^{\frac{T}{\mu}} \right) \notag \\
  \ifarx =&\: \lim_{\mu \to 0} \frac{T}{\mu} \left( k\log\left(1+\mu \delta\right) - (k-1) \log {\left( 1-\mu \delta \right)} \right) \notag \\ \fi
  =&\: kT \lim_{\mu \to 0} \frac{\log\left(1+\mu \delta\right)}{\mu} - (k-1)T \lim_{\mu \to 0}\frac{\log {\left( 1-\mu \delta \right)}}{\mu}
\end{align}
We examine the fraction inside the limit more closely. Since both the numerator and denominator of the fraction approach zero as \( \mu \to 0 \), we apply L'H\^{o}pital's rule:
\begin{align}
  \lim_{\mu \to 0}  \frac{ \log   \left( 1 \pm \mu \delta \right)}{\mu} = \lim_{\mu \to 0}  \frac{\pm \delta}{1\pm \mu \delta} = \pm \delta
\end{align}
Hence, we find:
\begin{align}
  \lim_{\mu \to 0} {\left( \frac{{(1+\mu \delta)}^k}{{\left(1-{\mu \delta}\right)}^{k-1}} \right)}^{\frac{T}{\mu}} = e^{kT\delta + (k-1)T \delta} = e^{-T \delta + 2k T \delta}
\end{align}

\section{Proof of Theorem~\ref{TH:DESCENT_RELATION}}\label{ap:descent_relation}
\noindent{We begin with~\eqref{eq:lipschitz_before_exp} and take expectations conditioned on \( \bcw_{i-1} \) to obtain:}
\begin{align}
	&\:\E \left \{ J(\w_{c, i}) | \bcw_{i-1} \right \} \notag \\
	\stackrel{(a)}{\le}&\: J(\w_{c, i-1}) - \mu {\left \| {\nabla J(\w_{c, i-1})} \right \|}^2 - \mu {\nabla J(\w_{c, i-1})}^{\T} \boldsymbol{d}_{i-1} \notag \\
	&\:+ \mu^2 \frac{\delta}{2} {\left \| {\nabla J(\w_{c, i-1})} + \boldsymbol{d}_{i-1} \right \|}^2 + \mu^2 \frac{\delta}{2} \E \left \{{\left \| \s_i \right \|}^2 | \bcw_{i-1} \right \} \notag \\
	\stackrel{(b)}{\le}&\: J(\w_{c, i-1}) - \mu {\left \| {\nabla J(\w_{c, i-1})} \right \|}^2 + \frac{\mu}{2} {\left \| \nabla J(\w_{c, i-1})\right \|}^{2} \notag \\
  &\:+ \frac{\mu}{2} {\left \| \boldsymbol{d}_{i-1} \right \|}^2 + \mu^2 \delta {\left \| {\nabla J(\w_{c, i-1})} \right \|}^2 + \mu^2 \delta {\left \| \boldsymbol{d}_{i-1} \right \|}^2 \notag \\
  &\:+ \mu^2 \frac{\delta}{2} \E \left \{{\left \| \s_i \right \|}^2 | \bcw_{i-1} \right \} \notag \\
	\stackrel{(c)}{\le}&\: J(\w_{c, i-1}) - \frac{\mu}{2}\left(1 - 2 \mu \delta \right) {\left \| {\nabla J(\w_{c, i-1})} \right \|}^2 \notag \\
  &\: + \frac{\mu}{2} \left( 1 + 2 \mu \delta \right) {\left \| \boldsymbol{d}_{i-1} \right \|}^2 + \mu^2 \frac{\delta}{2} \sigma^2\label{eq:interm_32444}
\end{align}
where cross-terms were removed in \( (a) \) due to the conditional zero-mean condition~\eqref{eq:conditional_zero_mean}, \( (b) \) follows from \( {\|a + b\|}^2 \le 2{\|a\|}^2 + 2{\|b\|}^2 \) and from \(-2a^{\T} b \leq \|a\|^2 + \|b\|^2\) and \( (c) \) is a result of grouping terms and Lemma~\ref{LEM:PERTURBATION_BOUNDS_FOURTH}.

{Note that~\eqref{eq:interm_32444} continues to be random due to the conditioning on \( \bcw_{i-1} \), but that it holds for every choice of \( \bcw_{i-1} \) with probability \( 1 \). Furthermore, since \( \w_{c, i-1} = \sum_{k=1}^N p_k \w_{k, i-1}\), the centroid \( \w_{c, i-1} \) is deterministic conditioned on \( \bcw_{i-1} \). As such, the event \( \w_{c, i-1} \in \mathcal{G} \) is deterministing conditioned on \( \bcw_{i-1} \), and~\eqref{eq:interm_32444} holds for every \( \w_{c, i-1} \in \mathcal{G} \). We can hence take expectations over \( \w_{c, i-1} \in \mathcal{G} \) and apply Lemma~\ref{LEM:CONDITIONING} to find:}
\begin{align}\label{eq:intermediate_323423423}
	&\:\E \left \{ J(\w_{c, i}) | \w_{c, i-1} \in \mathcal{G} \right \} \notag \\
  \le&\: \E \left \{ J(\w_{c, i-1}) | \w_{c, i-1} \in \mathcal{G} \right \}\notag \\
  &\: - \frac{\mu}{2}\left(1 - 2 \mu \delta \right) \E \left \{ {\left \| {\nabla J(\w_{c, i-1})} \right \|}^2 | \w_{c, i-1} \in \mathcal{G} \right \} \notag \\
	&\:+ \frac{\mu}{2} \left( 1 + 2 \mu \delta \right) \E \left \{ {\|\boldsymbol{d}_{i-1}\|}^2 | \w_{c, i-1} \in \mathcal{G} \right \} + \mu^2 \frac{\delta}{2} \sigma^2 \notag \\
  \stackrel{(a)}{\le}&\: \E \left \{ J(\w_{c, i-1}) | \w_{c, i-1} \in \mathcal{G} \right \} - \mu^2 c_1 \frac{c_2}{c_1} \left( 1+\frac{1}{\pi} \right) \notag \\
	&\:+ O(\mu) \E \left \{ {\|\boldsymbol{d}_{i-1}\|}^2 | \w_{c, i-1} \in \mathcal{G} \right \} + \mu^2 c_2 \notag \\
  \stackrel{(b)}{\le}&\: \E \left \{ J(\w_{c, i-1}) | \w_{c, i-1} \in \mathcal{G} \right \} - \mu^2 \frac{c_2}{\pi} \notag \\
	&\:+ \frac{\mu}{2} \left( 1 + 2 \mu \delta \right) \E \left \{ {\|\boldsymbol{d}_{i-1}\|}^2 | \w_{c, i-1} \in \mathcal{G} \right \}
\end{align}
{In step \( (a) \) we applied definition~\ref{DEF:SETS}, and in particular, that from~\eqref{eq:define_g} \( {\left \| \nabla J(\w_{c, i-1}) \right \|}^2 \ge \mu \frac{c_2}{c_1}\left(1+ \frac{1}{\pi}\right) \) whenever \( \w_{c, i-1} \in \mathcal{G} \), which implies:
\begin{equation}
  \E \left \{{\left \| \nabla J(\w_{c, i-1}) \right \|}^2 | \w_{c, i-1} \in \mathcal{G} \right \} \ge \mu \frac{c_2}{c_1}\left(1+ \frac{1}{\pi}\right)
\end{equation}
We also collected constants into \( c_1 \) and  \( c_2 \) defined in~\eqref{eq:define_c1}--\eqref{eq:define_c2} for brevity. Step \( (b) \) is obtained by grouping terms. Note that from lemma~\ref{LEM:PERTURBATION_BOUNDS_FOURTH}, we have a bound on \( \E {\|\boldsymbol{d}_{i-1}\|}^2 \), but not on the partial expectation conditioned over \( \w_{c, i-1} \in \mathcal{G} \). We can decompose the full expectation:}
\begin{align}
  &\: \E \left \{ {\|\boldsymbol{d}_{i-1}\|}^2 \right \} \notag \\
  =&\: \E \left \{ {\|\boldsymbol{d}_{i-1}\|}^2 | \w_{c, i-1} \in \mathcal{G} \right \} \cdot \pi_{i-1}^{\mathcal{G}} \notag \\
  &\: + \E \left \{ {\|\boldsymbol{d}_{i-1}\|}^2 | \w_{c, i-1} \in \mathcal{G}^C \right \} \cdot \pi_{i-1}^{\mathcal{G}^C} \stackrel{\eqref{eq:d_omufourth}}{\le} \: O(\mu^2)
\end{align}
which implies
\begin{equation}
  \E \left \{ {\|\boldsymbol{d}_{i-1}\|}^2 | \w_{c, i-1} \in \mathcal{G} \right \} \le \frac{O(\mu^2)}{\pi_{i-1}^{\mathcal{G}}}
\end{equation}
so that we obtain for~\eqref{eq:intermediate_323423423}:
\begin{align}
  &\:\E \left \{ J(\w_{c, i}) | \w_{c, i-1} \in \mathcal{G} \right \} \notag \\
  \le&\: \E \left \{ J(\w_{c, i-1}) | \w_{c, i-1} \in \mathcal{G} \right \} - \mu^2 \frac{c_2}{\pi} + \frac{O(\mu^3)}{\pi_{i-1}^{\mathcal{G}}}\label{eq:intermediate_bound_divide}
\end{align}
Similarly:
\begin{align}
	&\:\E \left \{ J(\w_{c, i}) | \w_{c, i-1} \in \mathcal{M} \right \} \notag \\
  \le&\: \E \left \{ J(\w_{c, i-1}) | \w_{c, i-1} \in \mathcal{M} \right \}\notag \\
  &\: - \frac{\mu}{2}\left(1 - 2 \mu \delta \right) \E \left \{ {\left \| {\nabla J(\w_{c, i-1})} \right \|}^2 | \w_{c, i-1} \in \mathcal{M} \right \} \notag \\
	&\:+ \frac{\mu}{2} \left( 1 + 2 \mu \delta \right) \E \left \{ {\|\boldsymbol{d}_{i-1}\|}^2 | \w_{c, i-1} \in \mathcal{M} \right \} + \mu^2 \frac{\delta}{2} \sigma^2 \notag \\
  \stackrel{(a)}{\le}&\: \E \left \{ J(\w_{c, i-1}) | \w_{c, i-1} \in \mathcal{M} \right \} + \mu^2 c_2 \notag \\
	&\:+ \frac{\mu}{2} \left( 1 + 2 \mu \delta \right) \E \left \{ {\|\boldsymbol{d}_{i-1}\|}^2 | \w_{c, i-1} \in \mathcal{M} \right \} \notag \\
  \stackrel{(b)}{\le}&\: \E \left \{ J(\w_{c, i-1}) | \w_{c, i-1} \in \mathcal{M} \right \} + \mu^2 c_2 + \frac{O(\mu^3)}{\pi_{i-1}^{\mathcal{M}}}
\end{align}
where \( (a) \) follows from the fact that \( {\left \| {\nabla J(\w_{c, i-1})} \right \|}^2 \ge 0\) with probability \( 1 \) and \( (b) \) made use of the same argument that led to~\eqref{eq:intermediate_bound_divide}.

\section{Proof of Lemma~\ref{LEM:DEVIATION_BOUNDS}}\label{AP:DEVIATION_BOUNDS}

We refer to~\eqref{eq:error_recursion}. Suppose \( i \le \frac{T}{\mu} \), where \( T \) is an arbitrary constant independent of \( \mu \). We then have for \( i \ge 0 \):
\begin{align}
  &\: \E \left \{ {\left \| \widetilde{\w}_{i+1}^{i^{\star}} \right \|}^2 | \boldsymbol{\mathcal{F}}_{i^{\star}+i} \right \} \notag \\
  \stackrel{\eqref{eq:error_recursion}}{=}&\: \E \Big \{ \Big \| \left( I - \mu \boldsymbol{H}_{i^{\star}+i} \right)  \widetilde{\w}_{i}^{i^{\star}}  + \mu \nabla J(\w_{c, i^{\star}}) \notag \\
  &\: \ \ \ \ \ \ \  + \mu \boldsymbol{d}_{i^{\star}+i} + \mu \s_{i^{\star}+i+1}\Big \|^2 | \boldsymbol{\mathcal{F}}_{i^{\star}+i} \Big \} \notag \\
  \stackrel{(a)}{=}&\: {\left \| \left( I - \mu \boldsymbol{H}_{i^{\star}+i} \right) \widetilde{\w}_{i}^{i^{\star}}+ \mu \nabla J(\w_{c, i^{\star}}) +\mu \boldsymbol{d}_{i^{\star}+i} \right \|}^2 \notag \\
  &\: + \mu^2 \E \left \{ {\left \| \s_{i^{\star}+i+1}\right \|}^2 | \boldsymbol{\mathcal{F}}_{i^{\star}+i} \right \} \notag \\
  \stackrel{(b)}{=}&\: \frac{1}{1-\mu \delta} {\left \|\left( I - \mu \boldsymbol{H}_{i^{\star}+i} \right) \widetilde{\w}_{i}^{i^{\star}}\right \|}^2 + \frac{\mu}{\delta} {\left \| \nabla J(\w_{c, i^{\star}}) + \boldsymbol{d}_{i^{\star}+i}\right \|}^2 \notag \\
  &\:+ \mu^2 \E \left \{ {\left \| \s_{i^{\star}+i+1}\right \|}^2 | \boldsymbol{\mathcal{F}}_{i^{\star}+i} \right \} \notag \\
  \stackrel{(c)}{=}&\: \frac{1}{1-\mu \delta} {\left \|\left( I - \mu \boldsymbol{H}_{i^{\star}+i} \right) \widetilde{\w}_{i}^{i^{\star}}\right \|}^2 + 2 \frac{\mu}{\delta} {\left \| \nabla J(\w_{c, i^{\star}}) \right \|}^2 \notag \\
  &\:+ 2 \frac{\mu}{\delta} {\left \| \boldsymbol{d}_{i^{\star}+i}\right \|}^2 + \mu^2 \E \left \{ {\left \| \s_{i^{\star}+i+1}\right \|}^2 | \boldsymbol{\mathcal{F}}_{i^{\star}+i} \right \} \notag \\
  \stackrel{(d)}{\le}&\: \frac{{(1+\mu \delta)}^2}{1-\mu \delta} {\left \| \widetilde{\w}_{i}^{i^{\star}}\right \|}^2 + 2 \frac{\mu}{\delta} {\left \| \nabla J(\w_{c, i^{\star}}) \right \|}^2 \notag \\
  &\:+ 2 \frac{\mu}{\delta} {\left \| \boldsymbol{d}_{i^{\star}+i}\right \|}^2 + \mu^2 \E \left \{ {\left \| \s_{i^{\star}+i+1}\right \|}^2 | \boldsymbol{\mathcal{F}}_{i^{\star}+i} \right \}
\end{align}
where \( (a) \) follows from the conditional zero-mean property of the gradient noise term in Assumption~\ref{as:gradientnoise}, \( (b) \) follows from Jensen's inequality
\begin{equation}
  {\|a + b\|}^2 \le \frac{1}{\alpha} {\|a\|}^2 + \frac{1}{1-\alpha} {\|b\|}^2\label{eq:jensens_second}
\end{equation}
with \( \alpha = \mu \delta < 1\) and \( (c) \) follows from the same inequality with \( \alpha = \frac{1}{2} \). Step \( (d) \) follows from the sub-multiplicative property of norms along with \( - \delta I \le \nabla^2 J(\w_{c, i^{\star}}) \le \delta I \), which follows from the Lipschitz gradient condition in Assumption~\ref{as:lipschitz}. {Since \( \w_{c, i^{\star}}  \) is deterministic conditioned on \( \boldsymbol{\mathcal{F}}_{i^{\star}+i} \) we can now take expectations over \(  \w_{c, i^{\star}} \in \mathcal{H} \) to obtain:}
\begin{align}
  &\: \E \left \{ {\left \| \widetilde{\w}_{i+1}^{i^{\star}} \right \|}^2 | \w_{c, i^{\star}} \in \mathcal{H} \right \} \notag \\
  \le&\: \frac{{(1+\mu \delta)}^2}{1-\mu \delta} \E \left \{ {\left \| \widetilde{\w}_{i}^{i^{\star}}\right \|}^2 | \w_{c, i^{\star}} \in \mathcal{H} \right \} \notag \\
  &\: + 2 \frac{\mu}{\delta} \E \left \{ {\left \| \boldsymbol{d}_{i^{\star}+i}\right \|}^2 | \w_{c, i^{\star}} \in \mathcal{H} \right \} \notag \\
  &\: + 2 \frac{\mu}{\delta} \E \left \{ {\left \| \nabla J(\w_{c, i^{\star}}) \right \|}^2 | \w_{c, i^{\star}} \in \mathcal{H} \right \} \notag \\
  &\:  + \mu^2 \E \left \{ {\left \| \s_{i^{\star}+i+1}\right \|}^2 | \w_{c, i^{\star}} \in \mathcal{H} \right \} \notag \\
  \stackrel{(a)}{\le}&\: \frac{{(1+\mu \delta)}^2}{1-\mu \delta} \E \left \{ {\left \| \widetilde{\w}_{i}^{i^{\star}}\right \|}^2 | \w_{c, i^{\star}} \in \mathcal{H} \right \} + 2 \frac{\mu}{\delta} \cdot \frac{O(\mu^2)}{\pi_{i^{\star}}^{\mathcal{H}}} \notag \\
  &\: + 2 \frac{\mu}{\delta} \cdot O(\mu) + O(\mu^2) \notag \\
  \le&\: \frac{{(1+\mu \delta)}^2}{1-\mu \delta} \E \left \{ {\left \| \widetilde{\w}_{i}^{i^{\star}}\right \|}^2 | \w_{c, i^{\star}} \in \mathcal{H} \right \} + O(\mu^2) + \frac{O(\mu^3)}{\pi_{i^{\star}}^{\mathcal{H}}}
\end{align}
where \( (a) \) follows from the perturbation bounds in Lemma~\ref{LEM:PERTURBATION_BOUNDS_FOURTH} and the starting assumption that \( \w_{c, i^{\star}} \) is an \( O(\mu) \)-square stationary point. {Note that, at time \( i=0 \), we have:
\begin{equation}
  \widetilde{\w}_{0}^{i^{\star}} = \w_{c, i^{\star}} - \w_{c, i^{\star} + 0} = 0
\end{equation}
and hence the initial deviation is zero, by definition.} Iterating, starting at \( i = 0 \) yields:
\begin{align}
  &\: \E \left \{ {\left \| \widetilde{\w}_{i}^{i^{\star}} \right \|}^2 | \w_{c, i^{\star}} \in \mathcal{H} \right \} \notag \\
  \le&\: \left( \sum_{n = 0}^{i-1} {\left( \frac{{(1+\mu \delta)}^{2}}{1 - \mu \delta} \right)}^n \right) \left( O(\mu^2) + \frac{O(\mu^3)}{\pi_{i^{\star}}^{\mathcal{H}}} \right) \notag \\
  =&\: \frac{1 - {\left( \frac{{(1+\mu \delta)}^{2}}{1 - \mu \delta} \right)}^i}{1 - {\frac{{(1+\mu \delta)}^{2}}{1 - \mu \delta} }} \left( O(\mu^2) + \frac{O(\mu^3)}{\pi_{i^{\star}}^{\mathcal{H}}} \right) \notag \\
  \ifarx =&\: \frac{\left({\left( \frac{{(1+\mu \delta)}^{2}}{1 - \mu \delta} \right)}^i - 1\right)\left(1 - \mu \delta \right)}{{1 + 2 \mu \delta + \mu^2 \delta^2 - 1 + \mu \delta }} \left( O(\mu^2) + \frac{O(\mu^3)}{\pi_{i^{\star}}^{\mathcal{H}}} \right) \notag \\ \fi
  =&\: \frac{\left({\left( \frac{{(1+\mu \delta)}^{2}}{1 - \mu \delta} \right)}^i - 1\right)\left(1-\mu \delta\right)}{{3 \delta + \mu \delta^2}} \left( O(\mu) + \frac{O(\mu^2)}{\pi_{i^{\star}}^{\mathcal{H}}} \right) \notag \\
  \le&\: \frac{\left({\left( \frac{{(1+\mu \delta)}^{2}}{1 - \mu \delta} \right)}^{\frac{T}{\mu}} - 1\right)\left(1-\mu \delta \right)}{{3 \delta + \mu \delta^2}} \left( O(\mu) + \frac{O(\mu^2)}{\pi_{i^{\star}}^{\mathcal{H}}} \right) \notag \\
  =&\: O(\mu) + \frac{O(\mu^2)}{\pi_{i^{\star}}^{\mathcal{H}}}
\end{align}
where the last line follows from Lemma~\ref{LEM:LIMITING_RESULTS} after noting that:
\ifarx \begin{align}
  &\: \frac{\left({\left( \frac{{(1+\mu \delta)}^{2}}{1 - \mu \delta} \right)}^{\frac{T}{\mu}} - 1\right)\left(1-\mu \delta \right)}{{3 \delta + \mu \delta^2}} \notag \\
  \le&\: \frac{\left({\left( \frac{{(1+\mu \delta)}^{2}}{1 - \mu \delta} \right)}^{\frac{T}{\mu}} - 1\right)\left(1-\mu \delta \right)}{{3 \delta}} \notag \\
  \ifarx \le&\: \frac{{\left( \frac{{(1+\mu \delta)}^{2}}{1 - \mu \delta} \right)}^{\frac{T}{\mu}} - {\left( \frac{{(1+\mu \delta)}^{2}}{1 - \mu \delta} \right)}^{\frac{T}{\mu}} \mu \delta - 1 + \mu \delta}{{3 \delta}} \notag \\ \fi
  \le&\: \frac{{\left( \frac{{(1+\mu \delta)}^{2}}{1 - \mu \delta} \right)}^{\frac{T}{\mu}}  - 1}{{3 \delta}}
\end{align}
\else \begin{align}
  \frac{\left({\left( \frac{{(1+\mu \delta)}^{2}}{1 - \mu \delta} \right)}^{\frac{T}{\mu}} - 1\right)\left(1-\mu \delta \right)}{{3 \delta + \mu \delta^2}}\le \frac{{\left( \frac{{(1+\mu \delta)}^{2}}{1 - \mu \delta} \right)}^{\frac{T}{\mu}}  - 1}{{3 \delta}}
\end{align} \fi
This establishes~\eqref{eq:ms_stability}. We proceed to establish a bound on the fourth-order moment.
Using the inequality~\cite{Sayed14}:
\begin{equation}
  \|a+b\|^4 \le \|a\|^4 + 3 \|b\|^4 + 8 \|a\|^2\|b\|^2 + 4 \|a\|^2 \left( a^{\T} b \right)
\end{equation}
we have:
\begin{align}
  &\: \E \left \{ {\left \| \widetilde{\w}_{i+1}^{i^{\star}} \right \|}^4 | \boldsymbol{\mathcal{F}}_{i^{\star}+i} \right \} \notag \\
  \ifarx \le&\: {\left \| \left( I - \mu \boldsymbol{H}_{i^{\star}+i} \right)  \widetilde{\w}_{i}^{i^{\star}} + \mu \nabla J(\w_{c, i^{\star}}) + \mu \boldsymbol{d}_{i^{\star}+i} \right \|}^4 \notag \\
  &\:+ 3 \mu^4 \E \left \{ \left \|   \s_{i^{\star}+i+1}\right \|^4 | \boldsymbol{\mathcal{F}}_{i^{\star}+i} \right \} \notag \\
  &\:+ 8 \mu^2 {\left \| \left( I - \mu \boldsymbol{H}_{i^{\star}+i} \right)  \widetilde{\w}_{i}^{i^{\star}} + \mu \nabla J(\w_{c, i^{\star}}) + \mu \boldsymbol{d}_{i^{\star}+i}\right \|}^2 \notag \\
  &\: \ \ \ \times \E \left \{ \left \|  \s_{i^{\star}+i+1}\right \|^2 | \boldsymbol{\mathcal{F}}_{i^{\star}+i} \right \} \notag \\
  &\:+ 4 \mu \left \| \left( I - \mu \boldsymbol{H}_{i^{\star}+i} \right)  \widetilde{\w}_{i}^{i^{\star}} + \mu \nabla J(\w_{c, i^{\star}}) + \mu \boldsymbol{d}_{i^{\star}+i} \right \|^2 \notag \\
  &\: \ \ \ \times {\left( \left( I - \mu \boldsymbol{H}_{i^{\star}+i} \right)  \widetilde{\w}_{i}^{i^{\star}}  + \mu \nabla J(\w_{c, i^{\star}}) + \mu \boldsymbol{d}_{i^{\star}+i}\right)}^{\T} \notag \\
  &\: \ \ \ \ \ \ \ \times \left( \E \left \{ \s_{i^{\star}+i+1} | \boldsymbol{\mathcal{F}}_{i^{\star}+1} \right \} \right) \notag \\ \fi
  \stackrel{(a)}{=}&\: {\left \| \left( I - \mu \boldsymbol{H}_{i^{\star}+i} \right)  \widetilde{\w}_{i}^{i^{\star}} + \mu \nabla J(\w_{c, i^{\star}}) + \mu \boldsymbol{d}_{i^{\star}+i} \right \|}^4 \notag \\
  &\:+ 3 \mu^4 \E \left \{ \left \|   \s_{i^{\star}+i+1}\right \|^4 | \boldsymbol{\mathcal{F}}_{i^{\star}+i} \right \} \notag \\
  &\:+ 8 \mu^2 {\left \| \left( I - \mu \boldsymbol{H}_{i^{\star}+i} \right)  \widetilde{\w}_{i}^{i^{\star}} + \mu \nabla J(\w_{c, i^{\star}}) + \mu \boldsymbol{d}_{i^{\star}+i}\right \|}^2 \notag \\
  &\: \ \ \ \times \E \left \{ \left \|  \s_{i^{\star}+i+1}\right \|^2 | \boldsymbol{\mathcal{F}}_{i^{\star}+i} \right \} \notag \\
  \stackrel{(b)}{=}&\: {\left \| \left( I - \mu \boldsymbol{H}_{i^{\star}+i} \right)  \widetilde{\w}_{i}^{i^{\star}} + \mu \nabla J(\w_{c, i^{\star}}) + \mu \boldsymbol{d}_{i^{\star}+i} \right \|}^4 + O(\mu^4) \notag \\
  &\:+ {\left \| \left( I - \mu \boldsymbol{H}_{i^{\star}+i} \right)  \widetilde{\w}_{i}^{i^{\star}} + \mu \nabla J(\w_{c, i^{\star}}) + \mu \boldsymbol{d}_{i^{\star}+i}\right \|}^2 O(\mu^2) \notag \\
  \stackrel{(c)}{=}&\: {\left \| \left( I - \mu \boldsymbol{H}_{i^{\star}+i} \right)  \widetilde{\w}_{i}^{i^{\star}} + \mu \nabla J(\w_{c, i^{\star}}) + \mu \boldsymbol{d}_{i^{\star}+i} \right \|}^4 + O(\mu^4) \notag \\
  &\:+ \Big( {\left \| \left( I - \mu \boldsymbol{H}_{i^{\star}+i} \right)  \widetilde{\w}_{i}^{i^{\star}} \right \|}^2 \notag \\
  &\: \ \ \ \ \ + \mu^2 {\left \| \nabla J(\w_{c, i^{\star}}) \right \|}^2 + \mu^2 {\left \| \boldsymbol{d}_{i^{\star}+i}\right \|}^2 \Big) O(\mu^2)
\end{align}
where in step \( (a) \) we dropped cross-terms due to the conditional zero-mean property of the gradient noise in Assumption~\ref{as:gradientnoise}, step \( (b) \) follows from the fourth-order conditions on the gradient noise in Assumption~\ref{as:gradientnoise} along with the perturbation bounds in Lemma~\ref{LEM:PERTURBATION_BOUNDS_FOURTH}, and \( (c) \) follows from Jensen's inequality, i.e. \( {\|a + b + c\|}^2 \le 3{\|a\|}^2 + 3{\|b\|}^2 + 3{\| c\|}^2 \). Taking expectations over \( \w_{c, i^{\star}} \in \mathcal{H} \) on both sides and collecting constant factors along with \( \mu \) in appropriate \( O(\cdot) \) terms:
\ifarx \begin{align}
  &\: \E \left \{  {\left \| \widetilde{\w}_{i+1}^{i^{\star}} \right \|}^4 | \w_{c, i^{\star}} \in \mathcal{H} \right \} \notag \\
  \le&\: \E \bigg \{ \Big \| \left( I - \mu \boldsymbol{H}_{i^{\star}+i} \right)  \widetilde{\w}_{i}^{i^{\star}}   + \mu \nabla J(\w_{c, i^{\star}}) \notag \\
  &\: \ \ \ \ \ \ \ + \mu \boldsymbol{d}_{i^{\star}+i}\Big \|^4 | \w_{c, i^{\star}} \in \mathcal{H} \bigg \} + O(\mu^4) \notag \\
  &\:+ \Big( \E \left \{ {\left \| \left( I - \mu \boldsymbol{H}_{i^{\star}+i} \right)  \widetilde{\w}_{i}^{i^{\star}}  \right \|}^2 | \w_{c, i^{\star}} \in \mathcal{H} \right \}\notag \\
  &\: \ \ \ \ \ + \mu^2  \E \left \{ {\left \|\nabla J(\w_{c, i^{\star}}) \right \|}^2 | \w_{c, i^{\star}} \in \mathcal{H} \right \} \notag \\
  &\: \ \ \ \ \ + \mu^2 \E \left \{  {\left \| \boldsymbol{d}_{i^{\star}+i} \right \|}^2 | \w_{c, i^{\star}} \in \mathcal{H} \right \}  \Big) O(\mu^2) \notag \\
  \le&\: \E \bigg \{ \Big \| \left( I - \mu \boldsymbol{H}_{i^{\star}+i} \right)  \widetilde{\w}_{i}^{i^{\star}}   + \mu \nabla J(\w_{c, i^{\star}}) \notag \\
  &\: \ \ \ \ \ \ \ + \mu \boldsymbol{d}_{i^{\star}+i}\Big \|^4 | \w_{c, i^{\star}} \in \mathcal{H} \bigg \} + O(\mu^4) \notag \\
  &\:+ \Big( {(1+\mu \delta)}^2 \E \left \{ {\left \| \widetilde{\w}_{i^{\star}}^{i}\right \|}^2 | \w_{c, i^{\star}} \in \mathcal{H} \right \} \notag \\
  &\: \ \ \ \ \ + \mu^2  \E \left \{ {\left \|\nabla J(\w_{c, i^{\star}}) \right \|}^2 | \w_{c, i^{\star}} \in \mathcal{H} \right \} \notag \\
  &\: \ \ \ \ \ + \mu^2 \E \left \{  {\left \| \boldsymbol{d}_{i^{\star}+i} \right \|}^2 | \w_{c, i^{\star}} \in \mathcal{H} \right \} \Big) O(\mu^2) \notag \\
  \le&\: \E \bigg \{ \Big \| \left( I - \mu \boldsymbol{H}_{i^{\star}+i} \right)  \widetilde{\w}_{i}^{i^{\star}}   + \mu \nabla J(\w_{c, i^{\star}}) \notag \\
  &\: \ \ \ \ \ \ \ + \mu \boldsymbol{d}_{i^{\star}+i}\Big \|^4 | \w_{c, i^{\star}} \in \mathcal{H} \bigg \} + O(\mu^4) \notag \\
  &\:+ \left( {(1+\mu \delta)}^2 O(\mu) + \mu^2  O(\mu) + \mu^2 \frac{O(\mu^2)}{\pi_{i^{\star}}^{\mathcal{H}}} \right) O(\mu^2) \notag \\
  =&\: \E \bigg \{ \Big \| \left( I - \mu \boldsymbol{H}_{i^{\star}+i} \right)  \widetilde{\w}_{i}^{i^{\star}}   + \mu \nabla J(\w_{c, i^{\star}}) \notag \\
  &\: \ \ \ \ \ \ \ + \mu \boldsymbol{d}_{i^{\star}+i}\Big \|^4 | \w_{c, i^{\star}} \in \mathcal{H} \bigg \} + O(\mu^3) + \frac{O(\mu^6)}{\pi_{i^{\star}}^{\mathcal{H}}}
\end{align}
\else \begin{align}
  &\: \E \left \{  {\left \| \widetilde{\w}_{i+1}^{i^{\star}} \right \|}^4 | \w_{c, i^{\star}} \in \mathcal{H} \right \} \notag \\
  \le&\: \E \bigg \{ \Big \| \left( I - \mu \boldsymbol{H}_{i^{\star}+i} \right)  \widetilde{\w}_{i}^{i^{\star}}   + \mu \nabla J(\w_{c, i^{\star}}) \notag \\
  &\: \ \ \ \ \ \ \ + \mu \boldsymbol{d}_{i^{\star}+i}\Big \|^4 | \w_{c, i^{\star}} \in \mathcal{H} \bigg \} + O(\mu^3) + \frac{O(\mu^6)}{\pi_{i^{\star}}^{\mathcal{H}}}
\end{align} \fi
Finally, from Jensen's inequality, we find for \( 0 < \alpha < 1 \):
\begin{align}
  \|a + b\|^4 &=  \frac{1}{\alpha^3} \left \| a \right \|^4 + \frac{1}{{(1-\alpha)}^3} \left \| b \right \|^4 \label{eq:jensens_fourth}
\end{align}
and hence for \( \alpha = 1 - \mu \delta \) and \( 0 < \mu < \frac{1}{\delta} \):
\begin{align}
  &\:\E \bigg \{ \Big \| \left( I - \mu \boldsymbol{H}_{i^{\star}+i} \right)  \widetilde{\w}_{i}^{i^{\star}}   + \mu \nabla J(\w_{c, i^{\star}}) \notag \\
  &\: \ \ \ \ \ \ \ + \mu \boldsymbol{d}_{i^{\star}+i}\Big \|^4 | \w_{c, i^{\star}} \in \mathcal{H} \bigg \} \notag \\
  \stackrel{\eqref{eq:jensens_fourth}}{\le}&\: \frac{{(1+\mu \delta)}^4}{{\left(1-{\mu \delta}\right)}^3} \E \left \{ {\left \| \widetilde{\w}_{i}^{i^{\star}} \right \|}^4| \w_{c, i^{\star}} \in \mathcal{H} \right \} \notag \\
  &\: + \frac{\mu^4}{\mu^3 \delta^3} \E \left \{ {\left \|\nabla J(\w_{c, i^{\star}}) + \boldsymbol{d}_{i^{\star}+i}\right \|}^4| \w_{c, i^{\star}} \in \mathcal{H} \right \} \notag \\
  \ifarx \stackrel{\eqref{eq:jensens_fourth}}{\le}&\: \frac{{(1+\mu \delta)}^4}{{\left(1-{\mu \delta}\right)}^3} \E \left \{ {\left \| \widetilde{\w}_{i}^{i^{\star}} \right \|}^4 | \w_{c, i^{\star}} \in \mathcal{H} \right \} \notag \\
  &\: + 8 \frac{\mu}{\delta^3} \bigg( \E \left \{ {\left \|\nabla J(\w_{c, i^{\star}})\right \|}^4 | \w_{c, i^{\star}} \in \mathcal{H} \right \} \notag \\
  &\: \ \ \ \ \ \ \ \ \ \ + \E \left \{ {\left \|\boldsymbol{d}_{i^{\star}+i}\right \|}^4 | \w_{c, i^{\star}} \in \mathcal{H} \right \}\bigg) \notag \\ \fi
  \ifarx \le&\: \frac{{(1+\mu \delta)}^4}{{\left(1-{\mu \delta}\right)}^3} \E \left \{ {\left \| \widetilde{\w}_{i}^{i^{\star}} \right \|}^4 | \w_{c, i^{\star}} \in \mathcal{H} \right \} \notag \\
  &\: + 8 \frac{\mu}{\delta^3} \left( O(\mu^2) + \frac{O(\mu^4)}{\pi_i^{\mathcal{H}}} \right) \notag \\ \fi
  \le&\: \frac{{(1+\mu \delta)}^4}{{\left(1-{\mu \delta}\right)}^3} \E \left \{ {\left \| \widetilde{\w}_{i}^{i^{\star}} \right \|}^4 | \w_{c, i^{\star}} \in \mathcal{H} \right \} + O(\mu^3) + \frac{O(\mu^5)}{\pi_i^{\mathcal{H}}}
\end{align}
\ifarx \else where the last line follows after applying~\eqref{eq:jensens_fourth} to \( {\left \|\nabla J(\w_{c, i^{\star}}) + \boldsymbol{d}_{i^{\star}+i}\right \|}^4 \). \fi\ Hence,
\begin{align}
  &\: \E \left \{  {\left \| \widetilde{\w}_{i+1}^{i^{\star}} \right \|}^4 | \w_{c, i^{\star}} \in \mathcal{H} \right \} \notag \\
  \le&\: \frac{{(1+\mu \delta)}^4}{{\left(1-{\mu \delta}\right)}^3} \E \left \{ {\left \| \widetilde{\w}_{i}^{i^{\star}} \right \|}^4 | \w_{c, i^{\star}} \in \mathcal{H} \right \} + O(\mu^3) + \frac{O(\mu^5)}{\pi_i^{\mathcal{H}}}
\end{align}
Recall again that \( \widetilde{\w}_{0}^{i^{\star}} = 0 \) and therefore iterating yields:
\begin{align}
  &\: \E \left \{ {\left \| \widetilde{\w}_{i}^{i^{\star}} \right \|}^4 | \w_{c, i^{\star}} \in \mathcal{H} \right \} \notag \\
  \le&\: \left( \sum_{n = 0}^{i-1} {\left( \frac{{(1+\mu \delta)}^4}{{\left(1-{\mu \delta}\right)}^3} \right)}^n \right) \left( O(\mu^3) + \frac{O(\mu^5)}{\pi_i^{\mathcal{H}}} \right) \notag \\
  =&\: \frac{1 - {\left( \frac{{(1+\mu \delta)}^4}{{\left(1-{\mu \delta}\right)}^3} \right)}^i}{1 - \frac{{(1+\mu \delta)}^4}{{\left(1-{\mu \delta}\right)}^3}} \left( O(\mu^3) + \frac{O(\mu^5)}{\pi_i^{\mathcal{H}}} \right) \notag \\
  \ifarx =&\: \frac{\left( {\left( \frac{{(1+\mu \delta)}^4}{{\left(1-{\mu \delta}\right)}^3} \right)}^i - 1 \right){(1-\mu \delta)}^3}{{(1+\mu \delta)}^4 - {\left(1-{\mu \delta}\right)}^3} \left( O(\mu^3) + \frac{O(\mu^5)}{\pi_i^{\mathcal{H}}} \right) \notag \\ \fi
  \le&\: \frac{{\left( \frac{{(1+\mu \delta)}^4}{{\left(1-{\mu \delta}\right)}^3} \right)}^i - 1}{{(1+\mu \delta)}^4 - {\left(1-{\mu \delta}\right)}^3} \left( O(\mu^3) + \frac{O(\mu^5)}{\pi_i^{\mathcal{H}}} \right) \notag \\
  \stackrel{(a)}{\le}&\: \frac{{\left( \frac{{(1+\mu \delta)}^4}{{\left(1-{\mu \delta}\right)}^3} \right)}^i - 1}{O(\mu)}\left( O(\mu^3) + \frac{O(\mu^5)}{\pi_i^{\mathcal{H}}} \right) \notag \\
  \ifarx =&\: \left( {\left( \frac{{(1+\mu \delta)}^4}{{\left(1-{\mu \delta}\right)}^3} \right)}^i - 1 \right) O(\mu^2) \notag \\
  \le&\: \left( {\left( \frac{{(1+\mu \delta)}^4}{{\left(1-{\mu \delta}\right)}^3} \right)}^{\frac{T}{\mu}} - 1 \right) \left( O(\mu^2) + \frac{O(\mu^4)}{\pi_i^{\mathcal{H}}} \right) \notag \\ \fi
  \le&\: O(\mu^2) + \frac{O(\mu^4)}{\pi_i^{\mathcal{H}}}
\end{align}
where in \( (a) \) we expanded:
\begin{align}
  &\: {(1+\mu \delta)}^4 - {(1-\mu \delta)}^3 \notag \\
  =&\: 1 + 4 \mu \delta + O(\mu^2) - 1 + 3 \mu \delta - O(\mu^2) = O(\mu)
\end{align}
and the last step follows from Lemma~\ref{LEM:LIMITING_RESULTS}. This establishes~\eqref{eq:mf_stability}. Eq.~\eqref{eq:mt_stability} then follows from Jensen's inequality via:
\ifarx \begin{align}
  &\: \E \left \{ {\left \| \widetilde{\w}_{i}^{i^{\star}} \right \|}^3 | \w_{c, i^{\star}} \in \mathcal{H} \right \} \notag \\
  \le&\: {\left( \E \left \{ {\left \| \widetilde{\w}_{i}^{i^{\star}} \right \|}^4 | \w_{c, i^{\star}} \in \mathcal{H} \right \} \right)}^{3/4} \notag \\
  \le&\: {\left( O(\mu^2) + \frac{O(\mu^4)}{\pi_{i^{\star}}^{\mathcal{H}}} \right)}^{3/4} \notag \\
  \ifarx =&\: O(\mu^{3/2}) + \frac{O(\mu^3)}{{\left(\pi_{i^{\star}}^{\mathcal{H}}\right)}^{4/3}} \notag \\ \fi
  \le&\: O(\mu^{3/2}) + \frac{O(\mu^3)}{{\pi_{i^{\star}}^{\mathcal{H}}}}
\end{align}
\else \begin{align}
  &\: \E \left \{ {\left \| \widetilde{\w}_{i}^{i^{\star}} \right \|}^3 | \w_{c, i^{\star}} \in \mathcal{H} \right \} \le\: {\left( \E \left \{ {\left \| \widetilde{\w}_{i}^{i^{\star}} \right \|}^4 | \w_{c, i^{\star}} \in \mathcal{H} \right \} \right)}^{3/4} \notag \\
  \le&\: {\left( O(\mu^2) + \frac{O(\mu^4)}{\pi_{i^{\star}}^{\mathcal{H}}} \right)}^{3/4} \le\: O(\mu^{3/2}) + \frac{O(\mu^3)}{{\pi_{i^{\star}}^{\mathcal{H}}}}
\end{align} \fi
We now study the difference between the short-term model~\eqref{eq:long_term_recursive} and the true recursion~\eqref{eq:error_recursion}. We have:
\begin{align}
  &\: \w_{c, i^{\star}+i+1} - \w_{c, i^{\star}+i+1}' \notag \\
  \ifarx =&\: - \widetilde{\w}{}^{i^{\star}}_{i+1} + \widetilde{\w}'{}^{i^{\star}}_{i+1} \notag \\ \fi
  =&\: - \left( I - \mu \boldsymbol{H}_{i^{\star} + i} \right)  \widetilde{\w}_{i}^{i^{\star}}  - \mu {\nabla J} (\w_{c, i^{\star}}) - \mu \boldsymbol{d}_{i^{\star}+i} - \mu \s_{i^{\star}+i+1} \notag \\
  &\: + \left( I - \mu \nabla^2 J( \w_{c, i^{\star}}) \right)  \widetilde{\w}'{}^{i^{\star}}_{i}  + \mu {\nabla J} (\w_{c, i^{\star}}) + \mu \s_{i^{\star}+i+1} \notag \\
  =&\: - \left( I - \mu \boldsymbol{H}_{i^{\star} + i} \right)  \widetilde{\w}_{i}^{i^{\star}} - \mu \boldsymbol{d}_{i^{\star}+i}\notag + \left( I - \mu \nabla^2 J( \w_{c, i^{\star}}) \right)  \widetilde{\w}'{}^{i^{\star}}_{i}  \notag \\
  =&\: \left( I - \mu \nabla^2 J(\w_{c, i^{\star}}) \right) \left( \w_{c, i^{\star}+i} - \w_{c, i^{\star}+i}' \right) - \mu \boldsymbol{d}_{i^{\star}+i} \notag \\
  &\:  + \mu \left( \boldsymbol{H}_{i^{\star}+i} - \nabla^2 J( \w_{c, i^{\star}}) \right)  \widetilde{\w}_{i}^{i^{\star}}\label{eq:intermediate_13513}
\end{align}
Before proceeding, note that the difference between the Hessians in the driving term can be bounded as:
\begin{align}
  &\:\left \| \nabla^2 J(\w_{c, i^{\star}}) - \boldsymbol{H}_{i^{\star} + i} \right \| \notag \\
  \ifarx =&\: \left \| \nabla^2 J(\w_{c, i^{\star}}) - \int_0^1 \nabla^2 J\left( (1-t) \w_{c, i^{\star}+i} + t \w_{c, i^{\star}} \right) dt \right \| \notag \\ \fi
  {=}&\: \left \| \int_0^1 \left(\nabla^2 J(\w_{c, i^{\star}}) - \nabla^2 J\left( (1-t) \w_{c, i^{\star}+i} + t \w_{c, i^{\star}} \right)\right) dt \right \| \notag \\
  \stackrel{(a)}{\le}&\: \int_0^1 \left \| \nabla^2 J(\w_{c, i^{\star}}) - \nabla^2 J\left( (1-t) \w_{c, i^{\star}+i} + t \w_{c, i^{\star}} \right)\right \| dt  \notag \\
  \stackrel{(b)}{\le}&\: \rho \int_0^1 \left \| (1-t) \w_{c, i^{\star}} - (1-t) \w_{c, i^{\star}+i} \right \| dt  \notag \\
  =&\: \rho \left \| \widetilde{\w}_{i}^{i^{\star}} \right \| \int_0^1 (1-t) dt =\: \frac{\rho}{2} \left \| \widetilde{\w}_{i}^{i^{\star}} \right \|\label{eq:lipschitz_driving}
\end{align}
where \( (a) \) follows Jensen's inequality and \( (b) \) follows form the Lipschitz Hessian assumption~\ref{as:lipschitz_hessians}. Returning to~\eqref{eq:intermediate_13513} and taking norms yields:
\begin{align}
  &\:{\| \w_{c, i^{\star}+i+1} - \w_{c, i^{\star}+i+1}' \|}^2 \notag \\
  =&\: \Big \| \left( I - \mu \nabla^2 J(\w_{c, i^{\star}}) \right) \left( \w_{c, i^{\star}+i} - \w_{c, i^{\star}+i}' \right)  \notag \\
  &\: - \mu \boldsymbol{d}_{i^{\star}+i}\notag + \mu \left( \boldsymbol{H}_{i^{\star}+i} -  \nabla^2 J( \w_{c, i^{\star}}) \right)  \widetilde{\w}_{i}^{i^{\star}}  \Big \|^2 \notag \\
  \stackrel{(a)}{\le}&\: \frac{1}{1 - \mu \delta}{\left \| \left( I - \mu \nabla^2 J(\w_{c, i^{\star}}) \right) \left( \w_{c, i^{\star}+i} - \w_{c, i^{\star}+i}' \right) \right \|}^2  \notag \\
  &\: + \frac{\mu^2}{\mu \delta} {\left \|  \boldsymbol{d}_{i^{\star}+i} + \left( \boldsymbol{H}_{i^{\star}+i} - \nabla^2 J( \w_{c, i^{\star}}) \right)  \widetilde{\w}_{i}^{i^{\star}}  \right \|}^2 \notag \\
  \stackrel{(b)}{\le}&\: \frac{1}{1 - \mu \delta}{\left \| \left( I - \mu \nabla^2 J(\w_{c, i^{\star}}) \right) \left( \w_{c, i^{\star}+i} - \w_{c, i^{\star}+i}' \right) \right \|}^2  \notag \\
  &\: + 2 \frac{\mu}{\delta} \left ( {\left \| \boldsymbol{d}_{i^{\star}+i} \right \|}^2 + {\left \| \left( \boldsymbol{H}_{i^{\star}+i} - \nabla^2 J( \w_{c, i^{\star}}) \right)  \widetilde{\w}_{i}^{i^{\star}}  \right \|}^2 \right) \notag \\
  \stackrel{\eqref{eq:lipschitz_driving}}{\le}&\: \frac{{(1+\mu \delta)}^2}{1 - \mu \delta}{\left \| \w_{c, i^{\star}+i} - \w_{c, i^{\star}+i}' \right \|}^2  \notag \\
  &\: + 2 \frac{\mu}{\delta} \left ( {\left \| \boldsymbol{d}_{i^{\star}+i} \right \|}^2 + \frac{\rho}{2}{\left \| \widetilde{\w}_{i}^{i^{\star}} \right \|}^4 \right)
\end{align}
where \( (a) \) again follows from Jensen's inequality~\eqref{eq:jensens_second} with \( \alpha = 1 - \mu \delta \) and \( (b) \) follows from the same inequality with \( \alpha = \frac{1}{2} \). Taking expecations over \( \w_{c, i^{\star}} \in \mathcal{H} \) yields:
\begin{align}
  &\:\E \left \{ {\| \w_{c, i^{\star}+i+1} - \w_{c, i^{\star}+i+1}' \|}^2 | \w_{c, i^{\star}} \in \mathcal{H} \right \} \notag \\
  \le&\: \frac{{(1+\mu \delta)}^2}{1 - \mu \delta} \E \left \{ {\left \| \w_{c, i^{\star}+i} - \w_{c, i^{\star}+i}' \right \|}^2 | \w_{c, i^{\star}} \in \mathcal{H} \right \} \notag \\
  &\: + 2 \frac{\mu}{\delta} \E \left \{ \left \| \boldsymbol{d}_{i^{\star}+i} \right \|^2 | \w_{c, i^{\star}} \in \mathcal{H} \right \} \notag \\
  &\: + \frac{\rho \mu}{\delta} \E \left \{ \left \|  \widetilde{\w}_{i}^{i^{\star}} \right \|^4 | \w_{c, i^{\star}} \in \mathcal{H} \right \} \notag \\
  \stackrel{(a)}{\le}&\: \frac{{(1+\mu \delta)}^2}{1 - \mu \delta} \E {\left \| \w_{c, i^{\star}+i} - \w_{c, i^{\star}+i}' \right \|}^2 + O(\mu^3) + \frac{O(\mu^3)}{\pi_{i^{\star}}^{\mathcal{H}}}
\end{align}
where \( (a) \) follows from the bound on the network disagreement in Lemma~\ref{LEM:DEVIATION_BOUNDS}.

Since both the true and the short-term model are initialized at \( \w_{c, i^{\star}} \), we have \( \w_{c, i^{\star}+0} - \w_{c, i^{\star}+0}' = 0 \). Iterating and applying the same argument as above leads to:
\begin{align}
  \E {\| \w_{c, i^{\star}+i+1} - \w_{c, i^{\star}+i+1}' \|}^2 \le O(\mu^2) + \frac{O(\mu^2)}{\pi_{i^{\star}}^{\mathcal{H}}}
\end{align}
which is~\eqref{eq:model_deviation}. %

%
\bibliographystyle{IEEEbib}
\bibliography{nonconvex}

\end{document}